\documentclass{amsart}
\usepackage{amscd}
\usepackage{amssymb}
\usepackage[all, knot]{xy}
\usepackage[top=1.2in, bottom=1.2in, left=1.2in, right=1.2in]{geometry}
\xyoption{all}
\xyoption{arc}
\usepackage{hyperref}

\newcommand{\rperf}[2]{\operatorname{RPerf}(#1 \into #2)}

\def\op{\text{op}}
\def\can{\operatorname{can}}

\def\rrperf{\operatorname{RPerf}}

\def\tF{\widetilde{F}}

\def\naive{{\text{naive}}}
\def\exact{{}}

\def\cube#1#2#3#4#5#6#7#8{
& #5 \ar[rr] \ar[dl] \ar@{-}[d] && #6 \ar[dd] \ar[dl] \\
#1 \ar[rr] \ar[dd]  & \ar[d] & #2 \ar[dd] \\
& #7 \ar@{-}[r] \ar[dl] & \ar[r] & #8 \ar[dl] \\
#3 \ar[rr] && #4 \\
}

\def\TP{\operatorname{TPC}}

\def\HomMF{\operatorname{\uHom_{MF}}}

\def\rH{\operatorname{H}}
\def\rHom{\operatorname{Hom}}
\def\rExt{\operatorname{Ext}}

\def\coh{\operatorname{coh}}

\def\vf{\mathbf f}

\def\cU{\mathcal U}

\def\sing{{\operatorname{Sing}}}

\def\cone{\operatorname{cone}}

\def\cP{\mathcal P}
\def\cE{\mathcal E}
\def\cL{\mathcal L}

\def\cC{\mathcal C}

\def\cF{\mathcal F}
\def\cG{\mathcal G}
\def\cO{\mathcal O}

\def\cM{\mathcal M}
\def\cN{\mathcal N}

\def\L{\mathbf{L}}

\def\coker{\operatorname{coker}}

\def\Sing{\operatorname{Sing}}
\def\Proj{\operatorname{Proj}}

\def\sExt{\operatorname{\widehat{Ext}}}

\def\Spec{\operatorname{Spec}}

\def\map#1{{\buildrel #1 \over \lra}} 
 
\def\lra{\longrightarrow}
\def\into{\hookrightarrow}
\def\onto{\twoheadrightarrow}
\def\cH{\mathcal{H}}

\newcommand{\bP}{\mathbb{P}}

\newcommand{\A}{\mathbb{A}}
\newcommand{\bH}{{\mathbb{H}}}
\newcommand{\G}{\mathbb{G}}

\newcommand{\R}{{\mathbf{R}}}
\newcommand{\F}{\mathbb{F}}
\newcommand{\E}{\mathbb{E}}
\newcommand{\bF}{\mathbb{F}}
\newcommand{\bE}{\mathbb{E}}
\newcommand{\bG}{\mathbb{G}}

\newcommand{\Z}{\mathbb{Z}}
\newcommand{\N}{\mathbb{N}}

\newcommand{\fm}{{\mathfrak m}}

\numberwithin{equation}{section}

\theoremstyle{plain} 
\newtheorem{thm}[equation]{Theorem}
\newtheorem{introthm}{Theorem}
\newtheorem*{introthm*}{Theorem}
\newtheorem*{question}{Question}
\newtheorem{cor}[equation]{Corollary}
\newtheorem{lem}[equation]{Lemma}
\newtheorem{prop}[equation]{Proposition}

\theoremstyle{definition}
\newtheorem{defn}[equation]{Definition}

\newtheorem{ex}[equation]{Example}

\theoremstyle{remark}
\newtheorem{rem}[equation]{Remark}

\newtheorem*{ack}{Acknowledgements}

\def\TP{\operatorname{TPC}}

\def\HomMF{\operatorname{\uHom_{MF}}}

\def\rDsg{\mathsf{D}^{\mathsf{rel}}_{\mathsf{sg}}}
\def\Dsg{\mathsf{D}_{\mathsf{sg}}}
\def\Dsing{\Dsg}
\def\Db{\mathsf{D}^\mathsf{b}}

\def\D{\mathsf{D}}
\def\Perf{\operatorname{Perf}}
\def\coh{\operatorname{coh}}

\newcommand\bs{\boldsymbol}

\newcommand{\Ext}[4]{\operatorname{Ext}_{#2}^{#1}(#3,#4)}

\newcommand{\Hom}[3]{\operatorname{Hom}_{#1}(#2,#3)}

\newcommand{\supp}{\operatorname{supp}}

\newcommand{\brc}[2]{\lbrace \, #1 \, | \, #2 \rbrace}
\newcommand{\li}{ < \infty}

\newcommand{\xra}[1]{\xrightarrow{#1}}
\newcommand{\xla}[1]{\xleftarrow{#1}}
\newcommand{\ps}[1]{\mathbb{P}_{#1}^{\text{c}-1}}

\newcommand{\YQ}{\gamma}
\newcommand{\RY}{\beta}
\newcommand{\RQ}{\delta}

\newcommand{\pd}{\operatorname{pd}}

\def\uHom{\mathcal{H}\!\operatorname{om}}
\def\uExt{\mathcal{E}\!\operatorname{xt}}
\def\suExt{\widehat{\uExt}}

\begin{document}

\title{Matrix factorizations in higher codimension}
\author{Jesse Burke}
\address{Department of Mathematics\\ 
Universit\"at Bielefeld\\ 
33501 Bielefeld\\ 
Germany}
\email{jburke@math.uni-bielefeld.de}
\author{Mark E. Walker}
\address{Department of Mathematics \\
University of Nebraska\\
Lincoln, NE 68588
}
\email{mwalker5@math.unl.edu}

\begin{abstract} 
We observe that there is an equivalence between the singularity category of an
affine complete intersection and the homotopy category of matrix factorizations over a
related scheme. This relies in part on a theorem of Orlov. Using this
equivalence, we give a geometric
construction of the
ring of
cohomology operators, and a generalization of the theory of support
varieties, which we call stable support sets. We settle a question of
Avramov about which stable support
sets can arise for a given complete intersection ring. We also use
the equivalence to construct a projective resolution of a module over a complete intersection ring from a matrix
factorization, generalizing
the well-known result in the hypersurface case. 
\end{abstract}


\maketitle

\tableofcontents

\section{Introduction}
Given an element $f$ in a commutative local ring $Q$, a {\em matrix
  factorization} of $f$ is a pair of $n \times n$ matrices $(A, B)$ such
that $AB = f\cdot I_n= BA$. This construction was introduced by
Eisenbud to study modules over the factor ring $R = Q/(f)$. Indeed, in the case
$Q$ is regular local and $f \ne 0$, he showed \cite[Theorem 6.1]{Ei80} that
for any finitely generated
$R$-module $M$, the minimal
free resolution of the $d$-th
syzygy of $M$, where $d$ is the Krull dimension of $Q$, is of the form
\[ \cdots \to R^n \xra{B \otimes_Q R} R^n \xra{A \otimes_Q R} R^n \xra{B \otimes_Q R} R^n \to 0, \]
for some matrix factorization $(A,B)$. Thus,
matrix factorizations describe the category of finitely generated
$R$-modules, ``up to finite projective dimension.'' This was further clarified by Buchweitz in \cite{Bu87} where he noted that
there is an equivalence of categories between the homotopy category
of matrix factorizations and the quotient of the bounded derived category of finitely
generated $R$-modules by perfect complexes:
\begin{equation} \label{L826}
[MF(Q,f)] \xra{\cong}  \Db(R)/\Perf(R) =: \Dsg(R).
\end{equation}
This functor sends a matrix factorization $(A,B)$ to 
$\coker A$, regarded as an object of $\Dsg(R)$.
 Here, the \emph{homotopy category} of matrix
factorizations is defined analogously to the homotopy category of
complexes of modules. Also, a complex is \emph{perfect} if it is
isomorphic in $\Db(R)$ to a bounded complex of finitely generated
projective $R$-modules. We call $\Dsg(R)$ the
\emph{singularity category} of $R$, following \cite{MR2101296}.

In the equivalence \ref{L826}, the right-hand side
is well defined for any ring $R$.

\begin{question} \label{question}
What should the left-hand side of \eqref{L826} be when $R$
is a complete intersection --- i.e., when $R = Q/(f_1, \ldots, f_c)$ for
a regular ring $Q$ and a $Q$-regular sequence $f_1, \ldots, f_c$?
\end{question}

For a complete intersection $R = Q/(f_1, \ldots, f_c)$
we refer to $c$ as the codimension of $R$ (technically $c$ is the codimension of the \emph{presentation} $Q/(f_1, \ldots,
f_c)$ --- but we fix a presentation throughout); complete intersections of codimension 1 are called
\emph{hypersurfaces}. In this paper we propose an answer for
higher codimension complete intersections and investigate the
consequences. To state it we introduce some notation.
If $X$ is a scheme, $\cL$ a line bundle on $X$, and $W$ a regular global
section of $\cL$, then a \emph{matrix factorization} of
the  triple
  $(X, \cL, W)$ consists of a pair of locally free coherent sheaves $\cE_1,
\cE_0$ on $X$ and maps
\[ \cE_1 \xra{e_1} \cE_0 \xra{e_0} \cE_1 \otimes \cL\]
such that $e_0 \circ e_1$ and $(e_1 \otimes 1_\cL) \circ e_0$ are both
multiplication by $W$. The 
{\em homotopy
category of matrix
factorizations}, written $[MF(X, \cL, W)]$, has these matrix factorizations
as objects. To define the morphisms in this category one starts out as
in the affine case, but some adjustment is needed to deal with the higher cohomology
of the locally free sheaves involved. In Section \ref{setup} we recall the
full definition of morphisms and give details on the following:
\begin{introthm} \label{introthm1}
Let $Q$ be a regular ring of finite Krull dimension, $f_1, \ldots, f_c$ a
$Q$-regular sequence, and  $R = Q/(f_1, \dots, f_c)$. Let $X = \ps Q = \Proj( Q[T_1,
\ldots, T_c] )$,
$W = \sum_i f_i T_i \in \Gamma(X, \cO_X(1))$, and define
$$
Y = \Proj(Q[T_1, \dots, T_c]/(W)) \into X,
$$ 
to be the zero
subscheme of $W$. There are equivalences of
triangulated categories
\[
\xymatrix@C=4em{[MF(\ps Q, \cO(1), W)] \ar[r]_(.6){\coker}^(.6)\cong & \Dsg(Y)
  \ar[r]_\Phi^\cong & \Dsg(R).}
\]
\end{introthm}
The first equivalence of this theorem has been proven in various
contexts by various authors in
\cite{1101.4051, Polishchuk:2010ys, 1102.0261}. The version we use
here is from our
previous paper \cite{BW11a}. The
second equivalence is essentially due to Orlov
\cite[Theorem 2.1]{MR2437083}. In Appendix \ref{appendix}
we provide a slight generalization of \emph{loc}.\ \emph{cit}.\ removing the
assumption that $Q$ is equicharacteristic and, more significantly, that
$Q$ is regular. (When $Q$ is not regular, $\Dsg$ is replaced by the
``relative singularity category''.)
This allows us to state many results in the paper in a
``relative'' form in which $Q$ is not necessarily regular but rather that the modules involved have finite
projective dimension over $Q$.

The composition of the equivalences in Theorem \ref{introthm1}, which we
write 
\[
\Psi = \Phi \circ \coker: [MF(\ps Q, \cO(1), W)] \xra{\cong}\Dsg(R),
\] 
provides an answer to Question \ref{question}. 
We  note here that matrix
factorizations of $W$ were (perhaps first) used in the proof of
\cite[3.2]{AvBu00} and their relation to MCM $R$-modules is currently
being studied in work of Buchweitz, Pham, and Roberts --- see
\cite{MR2849508}.
Also related is work of Isik \cite{1011.1484}. In this paper
we show that the above
equivalence and non-affine matrix factorizations provide a new and useful perspective on many aspects of homological
algebra over complete intersection rings. These results are
naturally stated in terms of the \emph{stable $\rExt$-modules} of finitely generated $R$-modules $M,N$,
defined as
\[ 
\sExt_R^n(M,N) := \Hom {\Dsing(R)} {M[-n]} N, \, \text{ $n \in \Z$.}
\] 
These modules are also known as the {\em stable cohomology}, or the {\em Tate cohomology}, of $M$ and
$N$. When $R$ is Gorenstein it is a well known result of Buchweitz
that they may be computed
using a complete resolution of $M$; see \cite{Bu87} and also 
Appendix \ref{compl_resl_sec}. The term ``stable'' reflects the facts that $\sExt_R^n(M,N)$ is zero if
$M$ or $N$ has finite projective dimension and that there are natural 
isomorphisms $\Ext n R M N \cong \sExt_R^n(M,N)$ for $n \gg 0$. 

When $R$ is a hypersurface it
follows from
\eqref{L826} that there are natural isomorphisms \[\sExt_R^n(M,N) \cong \sExt_R^{n+2}(M,N)\] for all $n
\in \Z$. For complete
intersections of higher codimension the stable $\rExt$-modules are
not necessarily
two-periodic in this sense. 
They are, however,  given as the sheaf cohomology modules of certain
sheaves on $Y$ that do satisfy a ``twisted'' periodicity, as we now
explain.
For finitely generated $R$-modules $M$ and $N$, let $\cM = \beta_* \pi^*(M)$ and
$\cN = \beta_* \pi^*(N)$, where $\pi : \ps R \to \Spec(R)$ is the
canonical projection and $\beta: \ps R \into Y$ is the evident closed
immersion. (Here, we identify $M$ and $N$ with coherent sheaves on
$\Spec(R)$ as usual.) 
We show in 
Section \ref{sec_eis_ops} (see Corollary
\ref{isom_ext_global_sections_hom}
and Proposition \ref{prop:nat_trans_ext_stable}) 
that for $n \gg 0$ there are natural isomorphisms
\begin{equation}
\label{intro_eqation}
\begin{aligned}
\sExt_R^{n}(M,N) & \cong \Gamma( Y , \uExt_{\cO_Y}^n( \cM, \cN )) \\
\uExt_{\cO_Y}^{n}(\cM, \cN)(1) & \cong \uExt_{\cO_Y}^{n+2}(\cM, \cN),
\end{aligned}
\end{equation}
where $\uExt_{\cO_X}^n(\cM, \cN)$ denotes the coherent sheaf
satisfying 
$$
\uExt^n_{\cO_X}(\cM, \cN)_x \cong  \rExt^n_{\cO_{X,x}}(\cM_x,
\cN_x).
$$ 
Combining these, there is an integer $m \geq 0$
and isomorphisms
\begin{equation}
\begin{aligned}\label{global_sects_isom}
\sExt_R^{m+2k}(M,N) &\cong \Gamma( Y, \uExt^{m}_{\cO_Y}(\cM,
\cN)(k) )\quad \text{for all } k \geq
0 \\
\sExt_R^{m+2k + 1}(M,N) &\cong \Gamma( Y, \uExt^{m+1}_{\cO_Y}(\cM,
\cN)(k) )\quad \text{for all } k \geq
0.
\end{aligned}
\end{equation}
In other words,  the even (respectively, odd) $\rExt$-modules of $M$ and $N$ are given (in
high enough
degrees) by the graded components of the graded $Q[T_1, \dots, T_c]$-module associated to the
coherent $\cO_Y$-sheaf
$\uExt^m_{\cO_Y}(\cM, \cN)$ (respectively, $\uExt^{m+1}_{\cO_Y}(\cM, \cN)$).
If $c = 1$, then $Y = \Spec R$ and $\cL$ is the trivial bundle, recovering the
two-periodicity for hypersurfaces.

The isomorphisms \eqref{global_sects_isom} show that in high degrees there are natural  maps of
degree two on the stable $\rExt$-modules, given by 
multiplication by the variables $T_1, \dots, T_c$.
On the other hand, it is well known that 
\[
\Ext *  R M N =
\bigoplus_{n \geq 0} \Ext n R M N
\] 
is a
graded module over the polynomial ring $R[\chi_1, \ldots, \chi_c]$
(which is graded by setting $|\chi_i| = 2$), where the action of the $\chi_i$'s is induced by the {\em Eisenbud
  operators} \cite{Ei80}.

Theorem \ref{thm:Eisenbud} shows that these two families of operators coincide under
the isomorphisms $\sExt^n_R(M,N) \cong \Ext n R M N$, for $n \gg 0$:
\begin{introthm}
\label{introthm3}
Let $M$ and $N$ be finitely generated
$R$-modules. For $n \gg 0$ and $i = 1, \ldots, c$ there is an equality:
\[ \chi_i = T_i : \sExt_R^n(M,N) \to \sExt_R^{n+2}(M,N).\]
\end{introthm}
This statement of this Theorem and its consequences occupy Section
\ref{sec_eis_ops}. The (surprisingly delicate) proof makes up Section
\ref{operator_proof_section}. 

As outlined above, the equivalence \eqref{L826} was first proved using the fact that
the minimal free resolution of
a module over a local hypersurface is eventually given by a matrix
factorization. In Section \ref{sec_proj_res} we go in the opposite
direction and use the equivalence $\Psi$ to show: 

\begin{introthm}
\label{introthm2}
Let $M$ be a finitely
generated $R$-module and $\bE = (\cE_1 \to \cE_0 \to \cE_1(1))$ a matrix factorization of $(\ps
Q, \cO(1), W)$ such that \[\Psi(\bE) \cong M \in \Dsing(R).\] There is
an integer $n$, that depends on $\bE$, such that the $n$-th
syzygy of $M$ has a free resolution constructed from the sheaf
cohomology of twists of $\cE_1$ and $\cE_0$.
\end{introthm}

See Theorem
\ref{proj_res_thm} for a precise description of this
resolution. In Section \ref{sec:inverse} we show how to construct a
matrix factorization $\bE_M$ such that $\Psi( \bE_M ) \cong M$, using
a finite resolution of $M$ over $Q$ and a system of higher homotopies,
as in \cite[\S 7]{Ei80}. Using this explicit construction we also show
that one may describe $\Dsg(R)$ using graded matrix factorizations of
$W$ over the ring $Q[T_1, \ldots, T_c]$.

When $R$ is local, Avramov and Buchweitz in \cite{AvBu00} use the action of $R[\chi_1, \ldots, \chi_c]$
on $\Ext * R M N$ to define the notion of a \emph{support
  variety}. Given finitely generated $R$-modules $M$ and $N$,
their support variety is a closed subset of $\A^c_k$, which we write
as 
\[V_Q^\vf(M, N)^{AB} \subseteq \A^c_k = \Spec
k[\chi_1, \ldots, \chi_c],
\] 
where $k$ is
the residue field of $R$.

For $R = Q/(\vf)$ a (not necessarily local) complete intersection, we
define the \emph{stable support set} of a pair of $R$-modules $M$ and $N$ to be
\[ V_Q^\vf(M,N) := \supp \widetilde{ \Ext * R M N} \subseteq \ps R.\]
After establishing
various formal properties of matrix
factorizations in Section \ref{props_mf} which we will use, in Section
\ref{suppsec} we show the properties of
stable support sets listed in the theorem below. These generalize the properties of support varieties proven
in \cite{AvBu00}; see also \cite[5.1]{BIK2}.

\begin{introthm}
\label{introthm4}
For finitely generated  $R$-modules
$M$, $N$, $M'$, and $N'$, we have
\begin{enumerate}

\item $V_Q^\vf(M,N)^{AB}$ is the cone
of the closed fiber of $V_Q^\vf(M,N)$ when $R$ is local;
\item $V_Q^\vf(M,N) = \emptyset$ if
and only if $\Ext n R M N = 0$ for all $n \gg 0$;
\item $V_Q^\vf(M,N) \cap V_Q^\vf(M',N') = V_Q^\vf(M,N') \cap V_Q^\vf(M',N)$; and 
\item $V_Q^\vf(M,N) =V_Q^\vf(M,M) \cap V_Q^\vf(N,N) = V_Q^\vf(N,M)$.
\end{enumerate}
\end{introthm}
Moreover, it follows almost immediately from the definition that $$V_Q^\vf(M, N)
= V_Q^\vf(\Omega^n_R(M), N) = V_Q^\vf(M, \Omega^n_R(N))$$ for any $n
\in Z$, where $\Omega^n_R(M)$ is an $n$th syzygy of $M$. (If $n < 0$
then $\Omega^n_R(M)$ is defined if and only if there is a long exact
sequence $0 \to M \to P_0 \xra{\partial} \cdots \xra{\partial} P_{-n}$ with $P_i$
projective. In this case $\Omega^n_R(M) = P_{-n} / \partial(P_{-n+1})$.)

One can check that stable support sets are contained in the singular
locus of $Y$, which is a subset of $\ps R$; see \ref{supp_is_in_sing}.
Thus the following result shows
that every ``conceivable'' subset
of $\ps R$ is the stable support set of an $R$-module. This answers
in the affirmative a question suggested to us by Avramov, and it
generalizes \cite{Be07} and \cite{AvIy07}; see also \cite[7.11]{1105.4698}.

\begin{introthm} \label{introthm5}
For every
closed subset $C$ of the singular locus of  \[Y = \Proj
\left (Q[T_1, \dots, T_c]/(W)\right ),\] there exists an
$R$-module $M$ with $C = V_Q^\vf(M,M)$. 
\end{introthm}

This Theorem is proved in Section \ref{suppsec}. Also in that section
we show how notions of support for modules and
complexes over complete intersections defined  in
\cite{BIK08, 1105.4698} relate to stable support. 

The paper also contains two appendices: Appendix
\ref{appendix} gives the generalization of Orlov's theorem described
above and Appendix \ref{compl_resl_sec} shows that if a module $M$ has a
complete resolution in the sense of \cite{Ve06}, then this resolution
may be used to compute the stable $\operatorname{Ext}$-modules as
defined above. This is well known in the Gorenstein case, see
\cite{Bu87}, but we could not find the result in the literature in the
generality that we need it here.

As alluded to above, we have endeavored to state and prove the
results in this paper in a more general setting than discussed in this
introduction. In particular, we do not typically assume $Q$ is a
regular ring, but instead assume that $M$ (and occasionally $N$) has
finite projective dimension over $Q$. This leads to somewhat more
complicated statements than those listed in this introduction. 
To counteract this, we have also sought to include explicitly the ``simple case'' in which
$Q$ is regular, so that the reader who is only interested in this case
can find the results in a more pleasing and easy-to-understand format.

\begin{ack}
 It is a pleasure to thank Lucho Avramov, Ragnar-Olaf Buchweitz, David
  Eisenbud, Srikanth Iyengar, Daniel Murfet, and Greg Stevenson for helpful
  conversations about the topics in this paper.
\end{ack}

\section{Matrix factorizations of locally free sheaves and Orlov's
  Theorem}
\label{setup}
In this section, we summarize relevant results from our
previous paper \cite{BW11a} on matrix factorizations of locally free
sheaves and then show how a theorem of
Orlov \cite{MR2437083} relates these to affine complete
intersections.

\subsection{Matrix factorizations of locally free sheaves}
\label{mf_defn}
Throughout $X$ will denote a Noetherian separated scheme and $\cL$ a line
bundle on $X$. To simplify notation, even if $\cL$ is not very ample, for a quasi-coherent sheaf $\cG$ (or a
complex of such) on $X$ and integer $n$, we will write
$\cG(n)$ for $\cG \otimes_{\cO_X} \cL^{\otimes n}$, where $\cL^{\otimes -n} := \uHom_{\cO_X}(\cL^{\otimes n},
\cO_X)$ for $n \geq 1$. (Here, and elsewhere, $\uHom$ denotes the
sheaf-$\rHom$.)
In particular, $\cO(1) = \cL$.
Similarly, if $f$ is morphism of (complexes
of) quasi-coherent sheaves, then $f(1) = f \otimes
id_{\cL}$.  

\begin{defn}
 Let $W$ be a global section of $\cL$. A {\em matrix factorization} $\E =  (\cE_1 \xra{e_1} \cE_0 \xra{e_0}
  \cE_1(1))$ of the triple $(X, \cL, W)$ consists of a pair of
  locally free coherent sheaves $\cE_1, \cE_0$ on $X$ and morphisms $e_1: \cE_1 \to \cE_0$ and $e_0: \cE_0 \to \cE_1(1)$
  such that $e_0 \circ e_1$ and $e_1(1) \circ e_0$ are multiplication
  by $W$. A {\em strict morphism} of matrix factorizations from
  $(\cE_1 \to \cE_0  \to \cE_1(1))$ to 
  $(\cF_1 \to \cF_0  \to \cF_1(1))$ is a pair of maps $\cE_0 \to
  \cF_0, \cE_1 \to \cF_1$ causing the evident squares to 
  commute. Matrix factorizations and strict morphisms of such form a category which we write
$MF(X, \cL, W)_\exact$ or just $MF_\exact$ for short.
\end{defn}
The previous definition first appeared in
\cite{Polishchuk:2010ys}. The category $MF(X, \cL, W)_\exact$ is made
into an exact category
in the sense of Quillen \cite{MR0338129} by declaring a sequence of
strict maps to be
exact if it is so in both degrees.

\begin{defn}
\label{twisted_periodic_def} 
A \emph{twisted periodic complex of locally free coherent sheaves}
for $(X, \cL)$ 
is a chain complex $\cC$ of locally free coherent sheaves on $X$ together
with a specified isomorphism $\alpha: \cC[2] \xra{\cong} \cC(1)$, 
where we use the convention that $\cC[2]^i =
\cC^{i+2}$. The category $\TP(X,\cL)$ has twisted periodic
complexes as objects and a morphism is a chain map that commutes
with the isomorphisms in the evident sense. There is an equivalence
\[ \TP(X,\cL) \cong MF(X, \cL, 0)_\exact\]
given by sending $(\cC, \alpha)$ to $\cC^{-1} \xra{d}  \cC^0
\xra{\alpha^{-1} \circ \, d} \cC^{-1}(1)$.
\end{defn}

The motivating example of a twisted periodic complex is
the following:
\begin{defn}
\label{def_hom_complex}
 Let $\E =  (\cE_1 \xra{e_1} \cE_0 \xra{e_0} \cE_1(1))$ and $\F =
 (\cF_1 \xra{f_1} \cF_0 \xra{f_0} \cF_1(1))$ be matrix
 factorizations of $(X, \cL, W)$. We define the {\em mapping complex}
 of $\bE$ and $\bF$,
 written $\HomMF(\E, \F)$, to be the following twisted periodic complex of locally free sheaves:
\begin{equation*}
{\small
\ldots \xra{\partial^0(-1)}
\begin{matrix}
  \uHom(\cE_0, \cF_1) \\
  \oplus \\
  \uHom(\cE_1, \cF_0(-1))
\end{matrix}\xra{\partial^{-1}}
\begin{matrix}
  \uHom(\cE_0, \cF_0) \\
  \oplus \\
  \uHom(\cE_1, \cF_1)
\end{matrix}
\xra{\partial^{0}}
\left(\begin{matrix}
  \uHom(\cE_0, \cF_1) \\
  \oplus \\
  \uHom(\cE_1, \cF_0(-1))
\end{matrix}
\right)(1)
\xra{\partial^{-1}(1)} \ldots
}
\end{equation*}

Here $\uHom(\cE_0, \cF_0)
  \oplus
  \uHom(\cE_1, \cF_1)$ lies in degree $0$, and 
the differentials are given by
$$
\partial^{-1} = 
\begin{bmatrix}
(f_1)_* & -e_0^* \\
-e_1^* & (f_0)_* \\
\end{bmatrix}
\quad \text{ and } \quad
\partial^{0} = 
\begin{bmatrix}
(f_0)_* & e_0^* \\
e_1^* & (f_1)_* \\
\end{bmatrix},
$$
using the canonical isomorphisms 
$$
\uHom(\cE_i, \cF_j(1))  \cong
\uHom(\cE_i, \cF_j)(1)  \cong
\uHom(\cE_i(-1), \cF_j)
$$ 
and
$$
\left(\begin{matrix}
  \uHom(\cE_0, \cF_1) \\
  \oplus \\
  \uHom(\cE_1, \cF_0(-1))
\end{matrix}
\right)(1)
\cong
\begin{matrix}
  \uHom(\cE_0, \cF_1(1)) \\
  \oplus \\
  \uHom(\cE_1, \cF_0).
\end{matrix}
$$
One checks that $\partial^0 \circ \partial^{-1}$ and $\partial^{-1}(1)
\circ \partial^0$ are both zero, and hence 
$\HomMF(\bE, \bF)$ is in fact a twisted periodic complex.
\end{defn}

Note that there is an isomorphism
\[ \Hom {MF_\exact} \E  \F \cong Z^0( \Gamma(X,\HomMF(\E, \F)) ),\]
where $\Gamma(X,\HomMF(\E, \F))$ is the complex of abelian groups
obtained by applying the global sections functor 
degree-wise to $\HomMF(\E, \F)$, and $Z^0$ denotes the cycles in degree zero.

Following \cite{1101.4051,Polishchuk:2010ys} we define the {\em homotopy
category} associated to $MF(X, \cL, W)$, written $[MF(X, \cL, W)]$ or just $[MF]$, as follows. First, one defines the ``naive homotopy category'',
written $[MF]_\naive$, to have Hom-sets
$$
 \rHom_{[MF]_\naive}(\bE, \bF) = \rH^0\left ( \Gamma \left( X,
     \HomMF(\bE, \bF) \right) \right ).
$$
Thus ``homotopic''
morphisms in $MF$ are identified.
The category $[MF]_\naive$ is triangulated with shift functor and
triangles as given in e.g.\ \cite[2.5]{BW11a}. An object $\bE$ of $[MF]_\naive$ is {\em locally contractible} if for each $x
\in X$ the evident localization at $x$, written $\bE_x$, is isomorphic to zero in 
$$
[MF(\Spec \cO_{X,x}, \cL_x, W_x)]_\naive.
$$
The set of such objects forms a thick subcategory.
The homotopy category $[MF]$ is the Verdier quotient 
$$
[MF(X, \cL, W)] = \frac{[MF(X, \cL, W)]_\naive}{\text{locally
    contractible objects}}.
$$

We developed in \cite{BW11a} a more
explicit description of the Hom-sets of $[MF]$ when $X$ is projective
over a Noetherian ring. Recall that a scheme $X$ is projective over a ring $Q$
if there is a closed embedding $j: X \into \bP^m_Q$ for some $m \geq
0$. In this case, we say that $\cO_X(1) := j^* \cO_{\bP^m_Q}(1)$ is the
\emph{corresponding very ample line bundle} on $X$. 

To describe the construction, we first make a fixed choice of a
  finite affine open cover $\cU = \{U_1, \dots, U_m\}$ of $X$, and for
  any quasi-coherent sheaf $\cF$ on $X$, let $\Gamma(\cU, \cF)$ denote
  the usual Cech construction. Since $X$ is separated, the cohomology
  of the complex
  $\Gamma(\cU, \cF)$ gives the sheaf cohomology of $\cF$.
We  define
 $\Gamma(\cU, \HomMF(\E, \F))$ to be
 the total complex associated to the bicomplex
$$
0 \to \bigoplus_i \Gamma(U_i,\HomMF(\E,
  \F)) \to \bigoplus_{i<j} \Gamma(U_i \cap U_i,\HomMF(\E, \F)) \to \cdots
$$
given by applying the Cech construction degree-wise.
If $\G$ is another matrix factorization, there is an evident morphism
of chain complexes
\begin{equation}
\Gamma(\cU, \HomMF(\E, \F)) \otimes \Gamma(\cU, \HomMF(\F,
\G)) \to \Gamma(\cU, \HomMF(\E, \G))\label{eq:gammacu-hommfe-f}
\end{equation}
which is associative and unital.

We set \[ \bH^q(X, \HomMF(\E, \F)) = \rH^q( \Gamma(\cU, \HomMF(\E, \F))
)\] and define the category $[MF(X, \cL, W)]_\bH$ to have the same
objects as $MF(X,
\cL, W)_\exact$ with
morphisms
$$
\Hom {[MF]_{\bH}} \E \F := \bH^0 (X, \HomMF(\E, \F)).
$$
The composition maps in this category are the maps in homology induced
by \eqref{eq:gammacu-hommfe-f}. 

There is a functor $[MF]_\naive \to [MF]_\bH$ that is the identity on
objects. The maps on morphisms are given by the canonical
maps
\[ \rH^0\left ( \Gamma \left( X,
     \HomMF(\bE, \bF) \right) \right ) \to \bH^0 (X, \HomMF(\E, \F)). \]
This functor induces a functor $[MF] \to [MF]_\bH$ by \cite[3.6]{BW11a}.

\begin{thm} \cite[Theorem 4.2]{BW11a}
Let $X$ be a scheme that is projective over a Noetherian ring $Q$ and
$\cL = \cO_X(1)$ be  the corresponding very ample line
bundle on $X$. 
For any global section $W$ of $\cL$, the functor
$$
[MF(X, \cL,W)] \to
[MF(X, \cL,W)]_{\bH}
$$
is an equivalence. 
\end{thm}
In the rest of the paper $X$ will always be assumed projective over
an affine scheme and we write $[MF]$ for $[MF]_\bH$.

\subsection{Singularity category}
The {\em singularity category} of a 
scheme $Z$ is the Verdier quotient
$$
\Dsing(Z) := \Db(Z)/\Perf(Z),
$$
where $\Db(Z)$ is the bounded derived category of coherent
sheaves on $Z$ and $\Perf(Z)$ is the full subcategory consisting
of perfect complexes --- i.e., those complexes that are locally
quasi-isomorphic to bounded complexes of free modules of finite rank. This construction was introduced by Buchweitz \cite{Bu87}
in the case when $Z$ is affine
and rediscovered by Orlov \cite{MR2101296}.

We need the following generalization:
\begin{defn}
\label{defn_rperf}
Let $i: Z \into X$ be a closed immersion of finite flat dimension.
An object $\cF$ in $\Db(Z)$ is
\emph{relatively perfect} on $Z$ if
$i_* \cF$ is perfect on $X$. We write $\rperf Z X$ for the full
subcategory of $\Db( Z)$ whose objects are relatively perfect
on $X$.

Since $i$ has finite flat dimension, $\Perf(Z)$ is a thick subcategory
of $\rperf Z
X$. We define the {\em relative singularity category of} $i$ to be
the Verdier quotient
$$
\rDsg(Z \into X) := \frac{\rperf Z X}{\Perf(Z)}.
$$
The canonical functor
$$ 
\rDsg(Z \into X) \to \Dsing(Z)
$$
is fully faithful and we thus identify $\rDsg(Z \into X) $ with
a full subcategory of $\Dsing(Z)$.  (A different definition of
``relative singularity category'' is given by Positselski in \cite{1102.0261}. There
is a fully faithful functor from the version given here to
Positselski's version but in general the two need not coincide; see
\cite[6.9]{BW11a}.)
\end{defn}

If $X = \Spec Q$ is affine, so that $Z = \Spec R$ where $R = Q/I$, we write 
$\rDsg(Q \onto R)$ for $\rDsg(\Spec R \into \Spec Q)$.
In this case, a finitely generated $R$-module $M$ is in $\rDsg(Q
\onto R)$ if and only if it has finite projective dimension as a $Q$-module.

Recall that $\cL$ is a line bundle on $X$ and $W$ is a
global section of $\cL$. We now set $Z \into X$ to be the zero
subscheme of $W$ (i.e., the subscheme with ideal sheaf given
as the image of the map $W^*: \cL^* \to \cO_X$).

For an object $\E = (\cE_1 \xra{e_1} \cE_0 \xra{e_0}
\cE_1(1))$ of $MF(X, \cL, W)_\exact$,
we define the \emph{cokernel of $\bE$},  written $\coker(\bE)$, to be $\coker(e_1)$. Multiplication by $W$ on
$\coker(e_1)$ is zero and so we regard $\coker \bE$ as a
coherent sheaf on $Z$. When $W$ is a regular, i.e.\ the map $\cO_X \xra{W} \cL$ is
injective, there is a finite resolution of $i_* \coker (\bE)$ on $X$ by locally
free sheaves:
\[ 0 \to \cE_1 \xra{e_1} \cE_0 \to i_* \coker( \bE) \to 0, \]
and hence $i_* \coker( \bE)$ is perfect on $X$. Thus the image of
$\coker(\bE)$ in $\Dsing(Z)$ is in $\rDsg(Z \into
X)$. This assignment is clearly natural and by
\cite[3.12]{Polishchuk:2010ys} there is an induced triangulated functor
$$
\coker: [MF(X, \cL, W)] \to \rDsg(Z \into X)$$
$$\bE \mapsto \coker(\bE).
$$

The following was first proved by Polishchuk and Vaintrob
in the case when $X$ is
regular, see \cite[3.14]{Polishchuk:2010ys}, and there are analogues in
\cite{1101.5847,1101.4051,1102.0261}. The version below is \cite[6.3]{BW11a}.
\begin{thm}
\label{our_equiv}
Let $X$ be a scheme that is projective over a Noetherian ring of
finite Krull dimension,
$\cL = \cO_X(1)$ be the corresponding very ample line bundle, and $W$
be a
regular global section of $\cL$.
Define $i: Z \into
X$ to be the zero subscheme of $W$. Then the 
triangulated functor
$$
\coker: [MF(X, \cL, W)] \to \rDsg(Z \into X)
$$
is an equivalence. In particular, if $X$ is
regular, there is an equivalence of triangulated categories
$$
\coker: [MF(X, \cL, W)]
  \map{\cong} \Dsing(Z).
$$
\end{thm}

\begin{cor}
 \label{complexes_isom_to_mod_rel_sing}
 Let $X$, $Z$ and $i$ be as in Theorem \ref{our_equiv}. If $\cM$ is any object of $\rDsg(Z \into X)$,
 then there is a coherent sheaf $\cF$ on $Z$ and an isomorphism $\cM
 \cong \cF$ in the category $\rDsg(Z \into X)$. In fact, $\cF$ may be chosen so that
 $i_*(\cF)$ admits a resolution of length one by locally free coherent
 sheaves on $X$. 
\end{cor}

\begin{proof}
 Let $\bE$ be the image of $\cM$ under some inverse of $\coker$. Then
 $\cM \cong \coker \bE$, and $\coker \bE$ is a coherent sheaf that
 admits a resolution of length one by locally free coherent sheaves.
\end{proof}

\subsection{Orlov's Theorem}
\label{assumptions}
We now fix $Q$ to be a commutative Noetherian ring of
finite Krull dimension, $\vf = (f_1, \dots, f_c)$ a
$Q$-regular sequence, and $R = Q/(\vf)$. Set
\[
\ps Q = \Proj \left ( Q[T_1, \dots, T_c] \right ) \quad \text{and} \quad \ps R = \Proj \left ( R[T_1, \dots, T_c] \right )
\] 
where each $T_i$ has degree 1. Define $W = f_1 T_1 + \ldots + f_c T_c
\in \Gamma(\ps Q, \cO_{\ps Q}(1))$ and
\[ Y = \Proj \left ( Q[T_1,
\ldots, T_c]/W \right ) \into \ps Q. \]

The natural surjection $Q \onto R$ induces an inclusion
$\RQ: \ps R \into \ps Q$. This is locally a
complete intersection of codimension $c-1$
and factors through the
map $\YQ: Y \into \ps Q$. We have a commutative diagram of schemes
\begin{equation} \label{E621}
\xymatrix{\ps R \ar@/^1.5pc/[rr]^{\RQ} \ar[r]^(.6){\RY} \ar[d]_{\pi}& Y \ar[r]^(.3){\YQ} &
  \ps Q \ar[d] \\
 \Spec R \ar[rr] & & \Spec Q} 
\end{equation}
where the vertical arrows are the canonical proper maps
and each horizontal arrow is locally a complete intersection and thus has
finite flat dimension. Since $\RY$ is a finite map, there is a right
adjoint to $\RY_*$, see \cite[III.6]{MR0222093}, which we write as
$\RY^\flat: \Db(Y) \to \Db(\ps R)$.

The following was proved by Orlov in \cite[Theorem
2.1]{MR2437083} under the added assumption that $Q$ is regular and
equicharacteristic, i.e. contains a field. A proof of the version
below may be found in the appendix; see Theorem \ref{Orlov_thm}.
\begin{thm}\label{orlov}
The functor $\R \pi_* \RY^\flat: \Db(Y) \to \Db(R)$ induces an
equivalence of triangulated categories
$$
\Phi: \rDsg(Y \subset \ps Q)
\map{\cong} 
\rDsg(Q \onto R).
$$ 
In particular, if $Q$ is regular, we have an equivalence
$$
\Phi: \Dsg(Y) \cong \Dsg(R).
$$
\end{thm}

\begin{proof}
In the notation of Theorem \ref{Orlov_thm}, let $S = \Spec Q$,
let $\cE$ be a free $Q$-module of rank $c$, and let $s = (f_1, \ldots, f_c) \in
\Gamma( S, \cE) = Q^c$.
\end{proof}

As an immediate consequence of Theorems
\ref{our_equiv} and \ref{orlov} we have:

\begin{cor} \label{orlovcor}
There is an equivalence of triangulated categories
$$
\Psi = \Phi \circ \coker: [MF(\ps Q, \cO(1), W)] \map{\cong} \rDsg(Q \onto R)
$$  
such that
\[ \Psi( \bE ) = \R \pi_* \RY^\flat \coker \bE \in \Dsing(R).\] In particular, if $Q$ is regular, we have
an equivalence of triangulated categories
$$
\Psi\!:\![MF(\ps Q, \cO(1), W)] \map{\cong} \Dsg(R).
$$  
\end{cor}

The inverse equivalence $\rDsg(Q \onto R) \to \rDsg(Y \subset \ps Q)$
is induced by the functor $\R \RY_* \circ
\L \pi^* \cong \RY_* \circ \pi^*$, but the inverse of $\coker: [MF] \map{\cong} \rDsg(Y
\subset \ps Q)$ requires making choices, and hence so does an inverse
equivalence to $\Psi$.
We fix one such inverse equivalence $\Psi^{-1}: \rDsg(Q \onto R) \to
[MF]$, and for $M \in \rDsg(Q \onto R)$ we write $\bE_M \in [MF]$ for
$\Psi^{-1}(M)$. Recall that
specifying the inverse $\Psi^{-1}$ amounts to picking, for each object
$M$ of $\rDsg(Q \onto R)$, a matrix factorization $\bE_M$ together with an
isomorphism $\Psi(\bE_M) \map{\cong} M$ in $\rDsg(Q \onto R)$. 
In Section \ref{sec:inverse} we give one explicit way of
  constructing $\bE_M$ using a system of higher
  homotopies on a $Q$-free resolution of $M$.

\begin{defn} \label{stable_ext_defn} Let $A$ be any commutative Noetherian ring and let $M$ and $N$ be complexes of $A$-modules with bounded
  finitely generated homology. For $q \in \Z$, the {\em $q$-th stable
$\rExt$-module} of $M$ and $N$ is
$$
\sExt^q_A(M,N) := \rHom_{\Dsg(A)}(M, N[q]).
$$
\end{defn} 
The definition above first appeared in \cite{Bu87}
  where the ring $A$ was assumed to be
  Gorenstein. We show in Appendix \ref{compl_resl_sec} that, as in \cite{Bu87}, the stable
  $\operatorname{Ext}$-modules may be computed using a complete
  resolution of $M$ when such a resolution exists.

Note that for all $q \in \Z$ there is a natural map
\begin{equation}
\rExt^q_A(M,N) \to \sExt^q_A(M,N),\label{eq:normal_to_stable}
\end{equation}
induced by the triangulated functor $\Db(A) \to \Dsg(A)$. 

The next result follows immediately from Corollary \ref{orlovcor} and
the definition of morphisms in $[MF]$.
\begin{cor} \label{Cor830}
Let $M$ and $N$ be objects of $\rDsg( Q \onto R)$. For all $q \in \Z$ there is an 
isomorphism, natural in $M$ and $N,$
$$
\sExt^q_R(M,N) 
\cong
\bH^q(\ps Q, \HomMF(\bE_M, \bE_N)).
$$
In particular, if $Q$ is regular, such an isomorphism holds for all
finitely generated $R$-modules.
\end{cor}

Fix an object $\E = (\cE_1 \xra{e_1} \cE_0 \xra{e_0}
\cE_1(1))$ of $MF(\ps Q, \cO(1), W)$. We write $\YQ^*\bE$ for the following complex of locally
free coherent sheaves on $Y$:
\begin{equation}
\cdots \to \YQ^*\cE_0(-1) \map{\YQ^* e_0(-1)} \YQ^*\cE_1 \map{\YQ^* e_1}
\YQ^*\cE_0 \map{\YQ^* e_0} \YQ^*\cE_1(1) \map{\YQ^* e_1(1)} \YQ^*\cE_0(1) \to
\cdots
\label{mf_pullback_defn}
\end{equation}
where $\YQ: Y \into \ps Q$ is the natural inclusion. This
is a complex since multiplication by $W$ is the zero map on $Y$.

The following result is an analogue of Corollary
\ref{Cor830} which relaxes the assumption that
$N$ is perfect over $Q$.

\begin{prop} \label{orlovcor2} 
Let $M$ be an object in $\rDsg(Q \onto R)$ and let $N$ be any object in
$\Dsg(R)$. Set $\cN$ to be the image of
$\beta_* \pi^* N$
in $\Dsg(Y)$.
For all $q\in \Z$ there are
isomorphisms that are natural in both arguments:
\begin{align}
\nonumber
\sExt^q_R(M,N) 
&\cong \Hom {\Dsg(R)} {\Psi(\bE_M)} {N[q]} \\
&\cong \Hom {\Dsg(Y)} {\coker
  \bE_M} {\cN[q]} \label{isoms}\\ 
&\cong
\bH^q\left (Y, \uHom_{\cO_Y}(\YQ^* \bE_M, \cN) \right ).\nonumber
\end{align}
\end{prop}

\begin{proof}
By definition $\Psi( \bE_M ) = \R \pi_*
\beta^\flat \coker(\bE_M) \cong M$ in $\Dsg(R)$. This gives the first
isomorphism in \eqref{isoms}.

The functor $\beta_* \pi^*$ induces an equivalence $\rDsg(Q \onto R) \to \rDsg(Y
\into X)$. This is an inverse to $\R \pi_*
\beta^\flat$, so the natural map $\beta_* \pi^* \R \pi_*
\beta^\flat \coker(\bE_M) \to \coker(\bE_M)$ is an isomorphism in
$\Dsg(Y)$.
This gives
\begin{align*} 
\Hom {\Dsg(R)} {\Psi(\bE_M)[-q]} N &= \Hom {\Dsg(R)} {\R \pi_*
\beta^\flat \coker(\bE_M)[-q]} N \\
&\cong \Hom {\Dsg(Y)} {\beta_* \pi^* \R \pi_*
\beta^\flat \coker(\bE_M)[-q]} \cN \\
&\cong \Hom {\Dsg(Y)} {\coker(\bE_M)[-q]} \cN
\end{align*}
which is the second isomorphism in \eqref{isoms}. The third isomorphism is due to \cite[5.10]{BW11a}.
\end{proof}

\section{Eisenbud operators}
\label{sec_eis_ops}

We return to the notation and assumptions of \S \ref{assumptions}. For
each matrix factorization $\bE$ and any $i = 1, \ldots, c$, multiplication by $T_i
\in \Gamma(\ps Q, \cO_{\ps Q}(1))$ defines a map
$$
T_i^{\bE}: \bE \to \bE[2]
$$
 given by
\begin{equation}\label{def_mult_T}
\xymatrix{ \cE_1 \ar[d]_{T_i} \ar[r] &
 \cE_0 \ar[d]_{T_i} \ar[r] & \cE_1(1) \ar[d]_{T_i} \\
\cE_1(1) \ar[r] & \cE_0(1) \ar[r] & \cE_1(2).} 
\end{equation}
These maps determine natural transformations
$$
T_i: id_{[MF]} \to (-)[2]_{[MF]} \quad \text{for} \quad i = 1, \dots, c
$$
from the identity functor on $[MF]$ to the functor which sends $\bE$
to $\bE[2]$.

On the other hand, there is a family of natural transformations $t'_i :
id_{\Db(R)} \to (-)[2]_{\Db(R)}$, for $i = 1, \dots c,$
given by the Eisenbud operators \cite{Ei80}; the construction of these
maps is recalled in \S\ref{Eops} below. These descend to natural transformations
$$
t_i : id_{\Dsg(R)} \to (-)[2]_{\Dsg(R)} \quad \text{for} \quad  i = 1, \dots c.
$$ 
One expects that these two natural transformations
coincide via the equivalence $$\Psi: [MF] \map{\cong} \rDsg(Q \onto R)$$
  of Corollary \ref{orlovcor}
and this is precisely the content of the
following theorem.

\begin{thm} \label{thm:Eisenbud}
For each object $\bE$ of $[MF]$ there is a commutative diagram in $\Dsg(R)$
$$
\xymatrix{ \Psi(\bE) \ar[dr]_{t^{\Psi(\bE)}_i}
\ar[r]^{\Psi(T_i^{\bE})} &
  \Psi(\bE[2]) \ar[d]^\cong \\
  &\Psi(\bE)[2] 
}
$$
where the vertical map is the canonical
isomorphism associated to the triangulated functor $\Psi$.
\end{thm}

Our proof of this theorem requires the development of additional
machinery, and is contained in Section
\ref{operator_proof_section}. In this section we explore several
consequences of the theorem. The first is the following:

\begin{cor} \label{isom_ext_global_sections_hom_MF}
Let $M$ and $N$ be objects of $\rDsg( Q \onto R)$. The diagram
\[ 
\xymatrix{
\bH^0(\ps Q, \HomMF(\bE_M[-q], \bE_N)) 
\ar[rr]^{\phantom{XXXX}\cong} \ar[d]_(.45){T_i} && \sExt_R^{q}(M,N)
\ar[d]^(.45){\sExt^q_R(t_i, N)} \\
\bH^{0}(\ps Q, \HomMF(\bE_M[-q-2], \bE_N)) 
\ar[rr]^{\phantom{XXXX} \cong} && \sExt_R^{q+2}(M,N) \\
} 
\]
commutes for $i=1, \dots, c$, 
where the horizontal isomorphisms are given by Corollary \ref{Cor830},
the left vertical map is $\bH^0( \ps Q, T_i)$, where $T_i$ is multiplication by $T_i \in \Gamma(
\ps Q, \cO(1) )$ on the complex of coherent sheaves $\HomMF(\bE_M[-q], \bE_N)$, composed with the natural isomorphism 
$$\HomMF(\bE_M[-q], \bE_N)(1) \cong
\HomMF(\bE_M[-q], \bE_N)[2] \cong \HomMF(\bE_M[-q-2], \bE_N),
$$
and  
$t_i$ is the $i$-th Eisenbud operator.
\end{cor}


\begin{proof}
Let $T_i: \bE[-q-2] \to \bE[-q]$ be the map defined in
\eqref{def_mult_T}. Consider the diagram
\[\xymatrix{ \Hom {[MF]} {\bE_M[-q]} {\bE_N} \ar[r]^(.43){=} &
  \bH^0(\ps Q, \HomMF(\bE_M[-q], \bE_N)) \ar[r]^(.58)\cong
  \ar[d]_{\bH^q( \ps Q, \HomMF(T_i, \bE_N))} &
  \Hom {\Dsg(R)} {M[-q]}{N} \ar[d]_{\Hom {\Dsg(R)} {t_i}{N}} \\
\Hom {[MF]} {\bE_M[-q-2]} {\bE_N} \ar[r]^(.43){=} & \bH^0(\ps Q, \HomMF(\bE_M[-q-2], \bE_N)) \ar[r]^(.58)\cong &
  \Hom {\Dsg(R)} {M[-q-2]}{N}
}
\]
where the horizontal maps are induced by the functor $\Psi$.
The diagram commutes by Theorem \ref{thm:Eisenbud}.
One checks that the middle vertical map above is equal to the left vertical
map in the statement of the corollary.
\end{proof}

We wish to extend the previous result by dropping the assumption that $N$
is perfect over $Q$. This is best stated using a ``stable
$\operatorname{Ext}$ sheaf" introduced below. With an eye toward
future applications, we work at a level of greater
generality than needed presently.

\subsection{Stable Ext sheaf}
Let $X$ be a scheme that is projective over a Noetherian ring $A$ of
finite Krull dimension,
$\cL = \cO(1)$ the associated very ample line bundle, and $W$ a
regular global section of $\cL$. Let $\YQ: Z \into X$ be the embedding
of the zero
subscheme of $W$. Under these assumptions, by Theorem \ref{our_equiv} there is an
equivalence 
\[\coker:
[MF(X, \cL, W)] \xra{\cong} \rDsg(Z \into X).\]
We
fix an inverse equivalence of $\coker$ and for $\cM$ in $\rDsg(Z
\into X)$ we let $\bE_\cM$ denote the image of $\cM$ under this
inverse. 

\begin{defn} \label{ssExtdef}
Let $X, \cL, W,$ and $Z$ be as above. For $\cM$ in $\rDsg(Z
\into X)$, $\cN$ a bounded complex of coherent sheaves on $Z$, and
an integer $q \in \Z$, define
$$
\suExt^q_{\cO_Z}(\cM, \cN)=  \cH^q 
\uHom_{\cO_Z}(\YQ^* \bE_\cM, \cN),
$$
where $\YQ^* \bE_\cM$ is
the complex of locally free coherent sheaves on $Z$ defined in
\eqref{mf_pullback_defn}. 
\end{defn}

\begin{rem}
\label{stable_ext_hommf}
If $\cN$ is also in $\rDsg(Z \into X)$ then by \cite[5.2.3]{BW11a}
  there is a isomorphism $\YQ_* \uHom_{\cO_Z}(\YQ^* \bE_\cM, \cN) \cong \HomMF( \bE_M, \bE_N
).$
Thus for each $q \in \Z$ there is an isomorphism \[\YQ_* \suExt^q_{\cO_Z}(\cM, \cN) \cong
\cH^q \HomMF( \bE_M, \bE_N
).\]
\end{rem}

\begin{rem}
When $Z = \Spec A$ is affine, the definition above agrees with
    the previous definition of stable Ext given in
    \ref{stable_ext_defn} by Example \ref{mf_compl_resl_ex} and Lemma
    \ref{stable_ext_by_compr}.
\end{rem}

\begin{lem} \label{lem1222}
The rule $(\cM, \cN) \mapsto \uHom_{\cO_Z}(\YQ^* \bE_\cM, \cN)$
defines a functor from 
$\rDsg(Z\into X)^\op \times \Db(Z)$ to $\D(\coh Z)$, the unbounded derived
category of $\coh Z$. In particular, 
$\suExt^q_{\cO_Z}(-, -)$
is a functor
from 
$\rDsg(Z\into X)^\op \times \Db(Z)$  to $\coh Z$.
\end{lem}

\begin{proof} Recall $\D(\coh Z)$ may be defined as the category of
  complexes of
  coherent sheaves, $\cC(\coh Z)$, with quasi-isomorphisms inverted. Then
  $\Db(Z)$ is the full subcategory of $\D(\coh Z)$ with objects those
  complexes with bounded cohomology. For any
  fixed matrix factorization $\bE$, we have a functor from $\cC^b(Z)$ to
  $\cC(\coh Z)$ given by 
$$
\cN \mapsto \uHom_{\cO_Z}(\gamma^* \bE_M, \cN).
$$
Since each
component of $\gamma^* \bE_M$ is locally free and $\cN$ is bounded, 
this functor preserves quasi-isomorphisms and
hence induces a functor from $\Db(Z)$ to $\D(Z)$. Since the
construction is natural for strict morphisms of matrix factorizations,
we obtain a functor from $MF(X, \cL, W)^{\op}_\exact \times \Db(Z)$ to
$\D(\coh Z)$.

If $\bE \to \bF$ is locally a homotopy equivalence, then
so are $\gamma^* \bE \to \gamma^* \bF$ and 
$$
\uHom_{\cO_Z}(\gamma^* \bE_M, \cN) \to 
\uHom_{\cO_Z}(\gamma^* \bE_M, \cN),
$$
and hence the latter is a quasi-isomorphism.
It follows that we get an induced functor \[[MF(X, \cL, W)]^{\op}
\times \Db(Z) \to \D(\coh Z).\] Precomposing
with the chosen inverse of $\coker$ gives the result.
\end{proof}

\begin{prop}  \label{prop:1450}
Let $X, \cL, W,$ and $Z$ be as above. 
Let $\cM$ be an object of $\rDsg(Z
\into X)$ and $\cN$ be an object of $\Db(Z)$.
For all $y \in Z$ and $q \in \Z$ there are isomorphisms
$$
\begin{aligned}
\suExt^q_{\cO_Z}(\cM, \cN) & \cong \suExt^{q+2}_{\cO_Z}(\cM,
\cN)(-1) \, \text{ and}  \\
\suExt^q_{\cO_Z}(\cM, \cN)_y & \cong \sExt^q_{\cO_{Z,y}}(\cM_y,
\cN_y), \\
\end{aligned}
$$
that are natural in $\cM$
and $\cN$. For all $q$, there is  a map
$$
\rHom_{\Dsg(Z)}( \cM[-q], \cN) \cong \bH^q(Z,  \uHom_{\cO_Z} (\gamma^* \bE_\cM, \cN)) \to 
\Gamma(Z,\suExt^q_{\cO_Z}(\cM, \cN)),
$$
that is natural in $\cM$ and $\cN$. This map is an isomorphism when $q \gg
0$.
\end{prop}

\begin{proof}
The first isomorphism is due to the fact that $\uHom_{\cO_Z}(\YQ^*
\bE_\cM, \cN)$ is a twisted periodic complex; see Definition \ref{twisted_periodic_def}.

For the second isomorphism, we are in the context of Example
\ref{mf_compl_resl_ex} which shows that there is a complete resolution
$T \to P \to \cM_y$ with $T = \gamma^*(\bE_\cM)_y$. It
follows from Lemma \ref{stable_ext_by_compr} that $T$ can be used to compute
$\sExt_{\cO_{Z,y}}(\cM_y, \cN_y)$ and we have:
\begin{align*}\sExt_{\cO_{Z,y}}^q(\cM_y,
\cN_y) &\cong \rH^q \left ( \rHom_{\cO_{Z,y}} ( \gamma^*(\bE_\cM)_y, \cN_y )
\right )\\ &\cong \cH^q \left (\uHom_{\cO_Z}(
\gamma^*(\bE_\cM), \cN)\right )_y \\
&= \suExt^q_{\cO_Z}(\cM, \cN)_y.
\end{align*}

The last map in the statement of the proposition is an edge map in the spectral sequence
$$
E^{p,q}_2 = H^p(Z,  \suExt^q_{\cO_Z}(\cM, \cN)) 
\Longrightarrow 
\bH^{p+q}(Z,  \uHom_{\cO_Z} (\gamma^* \bE_M, \cN)).
$$
By the first isomorphism and the fact that $\cL$ is very ample, we obtain
$H^p(Z,  \suExt^q_{\cO_Z}(\cM, \cN)) = 0$ if $p >0$ 
and $q \gg 0$ by Serre's Vanishing Theorem, and hence the map is an isomorphism for $q \gg 0$.
\end{proof}

The following is the sought-after generalization of
Corollary \ref{isom_ext_global_sections_hom_MF}, which drops the requirement that
$N$ is perfect over $Q$. We return to the context and notation of \S \ref{assumptions}.

\begin{cor}
\label{isom_ext_global_sections_hom}
 Let $M$ be an object of $\rDsg( Q \onto R)$ and let $N$ be an object
 of $\Dsg(R)$. Let $\cM$ and $\cN$  be
the images of $\beta_* \pi^* M$ and $\beta_* \pi^* N$, respectively, in $\Dsg(Y)$.
For all $q \in \Z$ there is a natural map
\begin{equation}
\label{map_1619}
\sExt_R^{q}(M,N) \to 
\Gamma( Y , \suExt_{\cO_Y}^q( \cM, \cN )) 
\end{equation}
that makes
the following diagram commute:
\[ 
\xymatrix{
\sExt_R^{q}(M,N) \ar[d]^(.45){\sExt_R^q(t_i,N)} \ar[rr]
&&
\Gamma(Y,  \suExt_{\cO_Y}^{q}( \cM, \cN )) 
\ar[d]_(.45){\Gamma(Y,T_i)} \\
\sExt_R^{q+2}(M,N) \ar[r] &
\Gamma(Y,  \suExt_{\cO_Y}^{q+2}( \cM, \cN ))  
\ar[r]^(.45)\cong & \Gamma(Y,  \suExt_{\cO_Y}^{q}( \cM, \cN )(1)).
\\
} 
\]
Here $T_i$ is multiplication by $T_i
\in \Gamma(Y, \cO_Y(1))$ on the coherent sheaf $\suExt_{\cO_Y}^{q}(
\cM, \cN )$, the isomorphism on the lower-right is from Proposition \ref{prop:1450},
and  
$t_i$ is the $i$-th Eisenbud
operator. Moreover, for $q \gg 0$, the horizontal maps in this diagram
are isomorphisms.
\end{cor}

\begin{proof}
By Proposition \ref{orlovcor2} we have the following string of
isomorphisms:
\[ \Hom {\Dsg(R)} {M[-q]} N \xra{\cong} \Hom {\Dsg(R)}
  {\Psi(\bE_M)[-q]} {N} \xra{\cong}  \bH^q(Y, \uHom_{\cO_Y}(\YQ^*
  \bE_M, \cN) ).\]
By Theorem \ref{thm:Eisenbud} and the naturality of the maps in
Proposition \ref{orlovcor2} the following diagram is commutative:
\[
\tiny{\xymatrix{
\Hom {\Dsg(R)} {M[-q]} N \ar[d]^{\Hom {\Dsg(R)} {t_i} N}
\ar[r]_(.45){\cong} & \Hom {\Dsg(R)}
  {\Psi(\bE_M)[-q]} {N} \ar[d]^{\Hom {\Dsg(R)} {\Psi(T_i^\bE)} N}
  \ar[r]^\cong & \bH^q(Y, \uHom_{\cO_Y}(\YQ^*
  \bE_M, \cN) ) \ar[d]^{\bH^q(Y, \uHom_{\cO_Y}(\YQ^* T_i^\bE, \cN))} \\
\Hom {\Dsg(R)} {M[-q-2]} N \ar[r]_(.45){\cong} & \Hom {\Dsg(R)}
  {\Psi(\bE_M[-2])[-q]} {N} \ar[r]_{\cong} & \bH^q(Y, \uHom_{\cO_Y}(\YQ^*
  \bE_M[-2], \cN) ).
}}
\]
Now by Proposition
\ref{prop:1450} we have a natural map $$\bH^q \left ( Y,
  \uHom_{\cO_Y}(\YQ^* \bE_M, \cN) \right ) \to
\Gamma\left ( Y,  \suExt^q_{\cO_Y}\left (\coker(\bE_M), \cN\right) \right ),$$ noting that
$\bE_{\coker (\bE_M)} \cong \bE_M$. This gives a commutative diagram
\[
\xymatrix{ 
  \bH^q \left ( Y,
  \uHom_{\cO_Y}(\YQ^* \bE_M, \cN) \right ) \ar[r]^\cong
  \ar[d]^{\bH^q(Y, \uHom_{\cO_Y}(\YQ^* T_i^\bE, \cN))}& 
\Gamma\left ( Y,  \suExt^q_{\cO_Y}\left (\coker(\bE_M), \cN\right) \right ) \ar[d]^{\Gamma(Y,
  \suExt^q(\coker(T_i^\bE), \cN))} \\
\bH^{q} \left ( Y,
\uHom_{\cO_Y}(\YQ^* \bE_M[-2], \cN) \right ) \ar[r]_\cong &
\Gamma\left ( Y,  \suExt^q_{\cO_Y}\left (\coker(\bE_M)[-2], \cN\right) \right ).}
\]
One checks that the following diagram commutes
\[ \xymatrix{ \Gamma\left ( Y,  \suExt^q_{\cO_Y}\left (\coker(\bE_M), \cN\right) \right )\ar[r]_= \ar[d]^{\Gamma(Y,
  \suExt^q(\coker(T_i^\bE), \cN))} & \Gamma\left ( Y,  \suExt^q_{\cO_Y}\left (\coker(\bE_M), \cN\right) \right ) \ar[d]^{\Gamma(Y, T_i)} \\
\Gamma\left ( Y,  \suExt^q_{\cO_Y}\left (\coker(\bE_M)[-2], \cN\right) \right ) \ar[r]_\cong & 
\Gamma\left ( Y,  \suExt^q_{\cO_Y}\left (\coker(\bE_M), \cN\right)(1) \right )
}\]
where the isomorphism $\Gamma\left ( Y,  \suExt^q_{\cO_Y}\left (\coker(\bE_M)[-2], \cN\right) \right )\xra{\cong} 
\Gamma\left ( Y,  \suExt^q_{\cO_Y}\left (\coker(\bE_M), \cN\right)(1)
\right )$ is from Lemma \ref{prop:1450} and $T_i$ is multiplication by $T_i \in
\Gamma( Y, \cO(1) )$ on the coherent sheaf $\suExt^q_{\cO_Y}\left (\coker(\bE_M), \cN\right)$.

Finally, we claim that there is a natural isomorphism $\cM \cong \coker(
\bE_M )$ in $\rDsg(Y \subseteq \ps Q)$. Indeed, since $\Psi( \bE_M ) := \R \pi_*
\beta^\flat \coker(\bE_M) \cong M$ in $\rDsg(Q \onto R)$, we have that
\[
\R \beta_* \L \pi^* \R \pi_*
\beta^\flat \coker(\bE_M) \cong \R \beta_* \L \pi^* M = \cM \in \rDsg(Y
\into \ps Q).
\]
Since $\R \beta_* \L \pi^*$ induces an equivalence $\rDsg(Q \onto R) \to \rDsg(Y
\into \ps Q)$ that is inverse to $\R \pi_*
\beta^\flat$, the natural map $\R \beta_* \L \pi^* \R \pi_*
\beta^\flat \coker(\bE_M) \to \coker(\bE_M)$ is an isomorphism in
$\rDsg(Y \into \ps Q)$. Thus there exists a functorial isomorphism 
\[\cM := \R \beta_* \L \pi^* M \xra{\cong} \coker( \bE_M). \]
By the claim, there is a natural isomorphism $\suExt^q_{\cO_Y}(\coker(\bE_M),
\cN) \cong \suExt^q_{\cO_Y}(\cM, \cN)$. 

Piecing together the commutative squares and the isomorphism $\cM
\cong \coker \bE_M$ gives the commutative diagram in the statement of
the proposition.

That the horizontal morphisms are isomorphisms for $q \gg 0$ follows
from \ref{prop:1450}.
\end{proof}

We return again to the more general context that $X$ is a
Noetherian scheme that is projective over an affine scheme, $\cL$ is
the corresponding very ample line bundle, $W$ is a regular global section of $\cL$, and
$\gamma: Z \into X$ is the zero subscheme of $W$.
Recall that $\uExt_{\cO_Z}^q(\cM, \cN)$ denotes the quasi-coherent
sheaf satisfying 
$$
\uExt_{\cO_Z}^q(\cM, \cN)_z 
\cong
\rExt_{\cO_{Z,z}}^q(\cM_z, \cN_z)
$$ 
for all $z \in Z$.

\begin{prop}
  \label{prop:nat_trans_ext_stable}
For each $q \in \Z$, there is a natural transformation
\[ \eta_q: 
\uExt^q_{\cO_Z}( - , - ) \to \suExt^q_{\cO_Z}(
\can(-) , - ), \]
of functors from $\rrperf(Z \into X)^\op \times  \Db(Z)$ to
$\coh Z$,
where 
\[\can: \rrperf(Z \into X) \to \rrperf( Z \into X ) /
\Perf Z = \rDsg(Z \into X)\] is the localization functor. For all $y \in Z$, the induced map on stalks
$$
\rExt^q_{\cO_{Z,y}}(\cM_y, \cN_y) \cong \uExt^q_{\cO_Z}( \cM , \cN )_y \to \suExt^q_{\cO_Z}(
\cM , \cN)_y \cong
\sExt^q_{\cO_{Z,y}}(\cM_y, \cN_y) 
$$
is the map \eqref{eq:normal_to_stable}.
Moreover, for fixed $\cM \in
\rrperf(Z \into X)$ and $\cN \in \Db(Z)$, there exists 
$q_0 \geq 0$ such that $\eta_q$ is an isomorphism for all $q \geq q_0$.
\end{prop}

\begin{rem}
 Corollaries \ref{isom_ext_global_sections_hom} and
 \ref{prop:nat_trans_ext_stable} establish \eqref{intro_eqation} and \eqref{global_sects_isom} from
  the Introduction.
\end{rem}

In the case of an affine scheme, say $Z = \Spec A$ and $\cM = \widetilde{M}$ for an
$A$-module $M$, one may compute
$\sExt^*_A(M, - )$ using a complete resolution $T \to P \to M$ by
Lemma \ref{stable_ext_by_compr}. In this case the map $T \to P$ induces the natural transformation $\Ext * A M -
\to \sExt^*_A(M, -)$.

In the more general case, $\YQ^* \bE_\cM$ represents, in some sense,
the $T$ in a complete resolution of $\cM$ in $\rDsg( Z \into X)$. However there is not necessarily a locally
free resolution $\cP$ of $\cM$ and a comparison map $\YQ^* \bE_\cM \to
\cP$. Thus to define the natural transformation in Proposition
\ref{prop:nat_trans_ext_stable} requires more work than in the
affine case. Our approach is to construct a global version of the Eisenbud
operators.

We begin by proving that given a complex $\cM = (\cdots \to \cM_{j+1}
\to \cM_j \to \cdots)$ of coherent sheaves 
on $Z$ that is bounded to the right (i.e., $\cM_j = 0$ for $j \ll 0$), 
there  exists a sequence of maps of locally free coherent sheaves on $X$
$$
\cE = \, \cdots \to \cE_{j+1} \xra{d_{j+1}} \cE_j \xra{d_{j}}
\cE_{j-1} \to \cdots
$$
and maps $g_j: \cE_j \to \gamma_*\cM_j$ of coherent
sheaves on $X$, for all $j$, such that the following conditions hold:
(1) $\cE_j = 0$ for $j \ll 0$, (2) the square 
\begin{equation} \label{L216}
\xymatrix{
\cE_{j} \ar[r] \ar[d] & \cE_{j-1} \ar[d] \\
\gamma_* \cM_{j} \ar[r]  & \gamma_* \cM_{j-1}  \\
}
\end{equation}
commutes for all $j$, (3) the sequence
$\gamma^*(\cE)$ is a complex, and (4) the map 
$\gamma^*(\cE) \to \cM$ of complexes given by the adjoints
of the $g_j$'s is a
quasi-isomorphism. Note that $\cE$ itself need not be  complex.

It suffices to construct such sheaves and maps of sheaves so that (1) --
(3) and the following condition holds: (4') the canonical map from $\gamma^*\cE_{j}$ to the
pullback of 
\begin{equation} \label{L216b}
\xymatrix{
& \gamma^* \cE_{j-1} \ar[d] \\
\cM_j \ar[r] &\cM_{j-1} \\
}
\end{equation}
is surjective for all $j$. Indeed, an easy diagram chase shows that,
so long as $\gamma^* \cE$ is a complex, 
condition (4') implies that $\gamma^* \cE \to \cM$ is a
quasi-isomorphism.

We give a recursive construction of $\cE_j$, $d_j$ and $g_j$.  
Since $\cM_j = 0$ for $j \ll 0$, there is no problem starting the
construction. Suppose $\cE_j$, $d_j$,  and $g_j$ have been constructed for all
$j \leq n$ so that \eqref{L216} commutes, the composition of $\gamma^*\cE_j
\to \gamma^* \cE_{j-1} \to \gamma^* \cE_{j-2}$ is the zero map, and
the map from $\gamma^* \cE_j$ to the pullback of \eqref{L216b} is
surjective, for all $j \leq n$. 
Observe also that these conditions ensure that
$\gamma^* \cE_j \onto \cM_j$ is surjective for all $j \leq n$, since $\cM_j =
0$ for $j \ll 0$.

Now define  $\cF_{n+1}$ to be the coherent
sheaf fitting into the pullback square
$$
\xymatrix{
\cF_{n+1} \ar[r] \ar@{->>}[d] & \gamma^* \cE_{n} \ar@{->>}[d] \\
\cM_{n+1} \ar[r] &\cM_{n}. \\
}
$$
Then we can construct a locally free coherent
sheaf $\cE_{n+1}$ and maps so that 
$$
\xymatrix{
\cE_{n+1} \ar[r]^{d_{n+1}}  \ar@{->>}[d] & \cE_{n} \ar@{->>}[d]^{\text{canonical}} \\
\gamma_* \cF_{n+1} \ar[r] & \gamma_* \gamma^* \cE_{n}  \\
}
$$
commutes and the left vertical map is surjective, by choosing $\cE_{n+1}$ to be a locally free coherent sheaf mapping
onto the  pullback of the other three components of this square. Define $g_{n+1}$ as the composition
$\cE_{n+1} \onto \gamma_* \cF_{n+1} \to \gamma_* \cM_{n+1}$. One may
readily verify that the above conditions are now satisfied for all $j
\leq n+1$, completing the recursive construction.

Tensoring the exact sequence $0 \to \cL^{-1} \to \cO_X \to \gamma_* \cO_Z
\to 0$ with any locally free coherent sheaf $\cG$
on $X$ gives an exact sequence
$$
0 \to \cG \otimes \cL^{-1} \to \cG \to \gamma_* \gamma^* \cG \to 0.
$$
It follows that for every $j$, the composite map
$$
\cE_{j} \to \cE_{j-2}
$$
factors through a uniquely determined map
$$
\tilde{t}_j:  \cE_j \to \cE_{j-2} \otimes \cL^{-1}.
$$
Let $t_j = \gamma^*(\tilde{t}_j)$. By localizing and using the results of 
\cite[Section 1]{Ei80}, we see that the collection $t_j$ determines
a map of chain complexes
$$
t: \gamma^*(\cE) \to \gamma^*(\cE\otimes \cL^{-1})[-2].
$$

\begin{lem} \label{Lem1116}
For a complex $\cM$ on $Z$ with bounded coherent cohomology, choose  $\cE$ and define $t:
\gamma^*(\cE) \to \gamma^*(\cE \otimes \cL^{-1})[-2]$ as  above. Then
$t$ induces a natural transformation of functors from $\coh Z$ to $\coh Z$
$$
\chi^q_{(X, \cL, W)}: \uExt^q_{\cO_Z}(\cM, - )(1) \to
\uExt^{q+2}_{\cO_Z}(\cM, - ).
$$
Moreover, this map is natural in $\cM$ 
and it is the unique map such that, for all $y \in Z$, the induced map on stalks at $y$,
$$
\rExt^q_{\cO_{Z,y}}(\cM_y, \cN_y) \otimes (I/I^2)^* \to
\rExt^{q+2}_{\cO_{Z_y}}(\cM, \cN),
$$
is the Eisenbud operator arising from the surjection $\cO_{X,y} \onto
\cO_{Z,y}$ with kernel $I$. In particular, it is independent of the
choice of $\cE_\cdot$.
\end{lem}

\begin{rem} We have the exact sequence $0 \to \cL^{-1} \to \cO_X \to
  \gamma_* \cO_Z \to 0$, and hence the stalk of $\cL^{-1}$ at $y$ is
  identified with $I$ and the stalk of $\gamma^* \cL^{-1}$ is identified with $I/I^2$.

Now, $I/I^2$ is a free $\cO_{Z,y}$-module of rank $1$, and
picking a generator $f$ of the principle ideal $I$ determines a basis of $I/I^2$ and
hence yields the more customary form of the Eisenbud operator for a
hypersurface:
$$
\rExt^q_{\cO_{Z,y}}(\cM_y, \cN_y) \to
\rExt^{q+2}_{\cO_{Z_y}}(\cM_y, \cN_y).
$$
But, this map depends on the choice of $f$ --- any other generator of
$I$ has the form $uf$ for a unit $u$ of $\cO_{X,x}$ and the operator
obtained from the choice of $uf$ is given by $\overline{u}$ times the operator
arising from $f$. 
\end{rem}

\begin{proof}
We use the locally free resolution $\gamma^* \cE$ of $\cM$
to compute $\uExt^*_{\cO_Z}(\cM, \cN)$. The chain map $t:
\gamma^*(\cE) \to \gamma^*(\cE \otimes \cL^{-1})$ induces
$$
\chi^q_{(X, \cL, W)}: \uExt^q_{\cO_Z}(\cM, \cN)(1) \to
\uExt^{q+2}_{\cO_Z}(\cM, \cN)
$$
for every $q \geq 0$. It follows from the construction that, locally, this agrees with
the map constructed by Eisenbud. The naturality and uniqueness
assertions may also be verified locally and thus follow from
\cite[Section 1]{Ei80}.
\end{proof}

\begin{proof}[Proof of Proposition \ref{prop:nat_trans_ext_stable}.]
Recall that by definition
$\suExt^r_{\cO_Z}( \cM, \cN ) = \cH^r \uHom(\gamma^* \bE_\cM, \cN)$ where $\bE_\cM$ denotes the image
of $\cM$ under the fixed inverse
of the functor $\coker$.
For $ r \geq 1$, there is a natural isomorphism
$\cH^r \uHom(\gamma^* \bE_\cM, \cN) \cong
\uExt^r_{\cO_Z}(\coker \bE_\cM, \cN)$ since $\gamma^* \bE_\cM$ is a
locally free resolution of $\coker \bE_\cM$ by \cite[5.2.1]{BW11a}. There is an isomorphism
$\coker(\bE_{\cM}) \xra{\cong} \cM$ in $\rDsg(Z \subset X)$ that is natural
in $\cM$. Such an isomorphism is an equivalence class of diagrams
\begin{equation} \label{E1216}
  \cM \leftarrow \cG \to \coker(\bE_{\cM})
\end{equation}
in $\rrperf(Z \into X)$ such that the cones of both morphisms lie in
$\Perf(Z)$.

For  $\cP \in \Perf(Z)$ and any $\cN$, we have $\uExt_{\cO_Z}^q(\cP, \cN) = 0$
for $q \gg 0$.
It follows that, given $\cN$, there is an integer $r_0$ such that both arrows in
$$
\uExt^r_{\cO_Z} (\cM, \cN)  \xra{\cong} \uExt^r_{\cO_Z}(\cG, \cN) 
\xla{\cong} \uExt^r_{\cO_Z}(\coker(\bE_{\cM}), \cN)
$$
are isomorphisms for $r \geq r_0$. This determines a family of
isomorphisms
$$
\eta_r: \uExt^r_{\cO_Z}( \cM, \cN) \xra{\cong}  \suExt^r_{\cO_Z}( \cM, \cN), r
\geq r_0.
$$
The induced map on the stalk at any point $y$ is given by the usual map
from ordinary to stable $\rExt$-modules by \ref{stable_ext_by_compr}
and \ref{prop:1450}. This proves, in
particular, that for $r \geq r_0$, $\eta_r$ is natural in both
arguments and that it is independent of choices.

Finally, to complete the proof of the proposition, we extend the
definition of $\eta_r$ to all $r$ by using Lemma \ref{Lem1116},
defining $\eta_r$ to be the composition of
$$
\uExt^r_{\cO_Z}( \cM, \cN) 
\xra{\chi^{\circ N}}
\uExt^{r+2N}_{\cO_Z}( \cM, \cN)(-N)
\xra{\eta_{r+2N}}  
\suExt^{r+2N}_{\cO_Z}( \cM, \cN)(-N)
\xra{\cong}
\suExt^r_{\cO_Z}( \cM, \cN)
$$
for $N \gg 0$, where the isomorphism on the right is from Proposition \ref{prop:1450}.
To see that the definition of $\eta_r$ is independent of the choice of $N \gg 0$ 
and that the resulting map is natural, it suffices to localize at an
arbitrary point. Locally, these properties are seen by combining
Lemma \ref{Lem1116}, the fact that 
$\suExt^r_{\cO_Z}( \cM, \cN) \xra{\cong}
\suExt^{r+2N}_{\cO_Z}( \cM, \cN)(-N)$ is, locally, given by the
Eisenbud operator, and the fact that the canonical map from $\rExt$ to
$\sExt$ commutes with the Eisenbud operator. The latter property is
seen to hold by using a
complete resolution to compute the stable Ext modules.
\end{proof}

\subsection{Some additional corollaries}
Again we return to the context and notation of \S \ref{assumptions}. For
complexes of $R$-modules $M$ and $N$ with bounded and finitely
generated cohomology, we define the graded $R[T_1,
\dots, T_c]$ modules:
$$
\begin{aligned}
\Ext {ev} R M N & = \bigoplus_{n \geq 0} \Ext {2n} R M N  \\
\Ext {odd} R M N  & = \bigoplus_{n \geq 0} \Ext {2n+1} R M N,
\end{aligned}
$$
where $\Ext {2n} R M N$ and $\Ext {2n+1} R M N$ lie in degree $n$ and
the $T_i$'s act as $\Ext {n} R {t_i} N$ for $t_i$ the $i$-th
Eisenbud operator on $M$. 

When $M$ has finite projective
dimension over $Q$, Gulliksen proved in \cite{Gu74} that $\Ext {ev} R
M N$ and $\Ext {odd} R M N$ are finitely generated graded modules over the
   ring $R[T_1, \ldots, T_c]$. They therefore determine coherent
   sheaves on $\ps R$, which we write as
$\widetilde{    \Ext {ev} R M N}$ and $\widetilde{ \Ext {odd} R M
  N}$. Note that while Gulliksen did not use Eisenbud operators, it was
shown in \cite{AvSu98} that the
   action used by Gulliksen coincides, up to a sign, with the Eisenbud
   operators and moreover Gulliksen's result, which was originally stated only
   for modules, was extended to complexes of
   $R$-modules with bounded and finitely generated cohomology.

\begin{cor}
\label{sheafify_ext}
Let $M$ and $N$ be
complexes of $R$-modules with bounded and finitely
generated cohomology and assume that $M$ is perfect over $Q$. Set $\cM
= \beta_* \pi^* M$ and $\cN = \beta_* \pi^* N$, where $\beta : \ps R \into Y$ is the canonical inclusion and $\pi: \ps R
\to \Spec R$ is the canonical proper map. For $q \gg 0$,
there are isomorphisms of
coherent sheaves on $Y$:
\[ 
\begin{aligned}
\uExt_{\cO_Y}^{2q}( \cM, \cN)(-q) & \cong \RY_* \widetilde{ \Ext {ev} R M N
} \\
\uExt_{\cO_Y}^{2q+1}( \cM, \cN)(-q) & \cong \RY_* \widetilde{
    \Ext {odd} R M N}.
\end{aligned}
\]
\end{cor}

\begin{proof}
For $n \gg 0$, we have natural isomorphisms $\Ext n R M N \cong
\sExt^n_R(M,N)$ that commute with the Eisenbud operators. 
It thus follows from Corollary \ref{isom_ext_global_sections_hom} and
Proposition \ref{prop:nat_trans_ext_stable}
that, for $q \gg 0$, we have  isomorphisms
$$
\bigoplus_{n \geq q} \Ext {2n} R M N \cong
\bigoplus_{n \geq q} \rH^0( Y, \uExt^{2n}_{\cO_Y}(\cM, \cN) )
\cong
\bigoplus_{n \geq q} \rH^0( Y, \uExt^{2q}_{\cO_Y}(\cM, \cN)(n-q) )
$$ 
of graded $R[T_1, \ldots,
T_c]$-modules. The coherent sheaf on $Y$ associated to the 
first of these is $\beta_* \widetilde{ \Ext {ev} R M N}$ and the coherent
sheaf associated to the
last of
these is $\uExt^{2q}_{\cO_Y}(\cM, \cN)(-q)$.

The proof for the odd degree case is identical.
\end{proof}

\begin{rem} 
The results above recover Gulliksen's theorem on the
finiteness of $\Ext * R M N$ over $R[T_1, \ldots, T_c]$. Indeed, since
$\uExt^q_{\cO_Y}(\cM, \cN)$
is coherent, so is $\YQ_* \uExt^q_{\cO_Y}(\cM, \cN)$ where $\gamma: Y
\into \ps Q$ is the canonical inclusion. Thus the $Q[T_1, \ldots,
T_c]$-module 
\[\bigoplus_{n \geq 0} \rH^0( \ps Q, \YQ_* \uExt^d_{\cO_Y}(\cM, \cN)(n) )\] is finitely generated.
Hence so
is $(\Ext {ev} R M N)_{\geq q} $ for $q \gg 0$, by Corollary \ref{isom_ext_global_sections_hom} and
Proposition \ref{prop:nat_trans_ext_stable}. Adding on the finitely generated
$R$-module $(\Ext {ev} R M N)_{< q} $, we see that $\Ext {ev} R M
N$ is finitely generated over $Q[T_1, \ldots, T_c]$. Finally, since
$\Ext {ev} R M N$ is annihilated by $f_1, \ldots, f_c$, it is finitely
generated over $R[T_1, \ldots, T_c]$.
\end{rem}

\begin{cor} \label{suppcor}
Let $M$ and $N$ be
complexes of $R$-modules with bounded and finitely
generated cohomology and assume that $M$ is perfect over $Q$. Let $\cM
= \beta_* \pi^* M$ and
$\cN = \beta_* \pi^* N$. For all $n \in \Z$, the support of the $\cO_Y$-module
$\suExt^n_{\cO_Y}(\cM, \cN)$ is contained in $\ps R$, and
we have equalities 
$$
\begin{aligned}
\supp \suExt^{0}_{\cO_Y}(\cM, \cN) & =
\supp \widetilde{ \Ext {ev} R M N } \\
\supp \suExt^{1}_{\cO_Y}(\cM, \cN)& =
\supp \widetilde{ \Ext {odd} R M N }
\end{aligned}
$$
of closed subsets of $\ps R$.
\end{cor}

\begin{proof}
For $n \gg 0$ we have that $\suExt^{n}_{\cO_Y}(\cM, \cN) \cong
\uExt^{n}_{\cO_Y}(\cM, \cN)$ by Proposition
\ref{prop:nat_trans_ext_stable}. For $q \gg 0$, by Corollary \ref{sheafify_ext}
we have that $\uExt_{\cO_Y}^{2q}( \cM, \cN)(-q) \cong \RY_*
\widetilde{ \Ext {ev} R M N}$ and
$\uExt_{\cO_Y}^{2q+1}( \cM, \cN)(-q) \cong \RY_* \widetilde{
    \Ext {odd} R M N}$. Finally, by \ref{prop:1450},
  $\suExt_{\cO_Y}^{2q}( \cM, \cN)(-q) \cong \suExt_{\cO_Y}^{0}( \cM,
  \cN)$ and similarly for the odd case.
\end{proof}

We close with a definition.
\begin{defn}  \label{suppdef}
Let $M$ and $N$ be
complexes of $R$-modules with bounded and finitely
generated cohomology and assume that $M$ is perfect over $Q$.
The \emph{stable support set} of $(M,N)$ 
is the closed subset $V_Q^{\vf}(M, N)$ of $\ps R$ defined as
$$
V_Q^{\vf}(M,N) := 
\supp \widetilde{ \Ext {ev} R M N }
\cup
\supp \widetilde{ \Ext {odd} R M N }.
$$

The \emph{stable support set} of $M$ is defined to
be $V_Q^{\vf}(M) := V_Q^{\vf}(M, M)$. 
\end{defn}
In Section \ref{suppsec} we study this support further and
clarify its relation to the notions of support for $R$-modules defined
previously in
\cite{AvBu00,BIK08,1105.4698}. 

\begin{rem} \label{supp_rem}
By Corollary \ref{suppcor}, there is an equality
$$
V_Q^{\vf}(M,N) = 
\supp \suExt^{0}_{\cO_Y}(\cM, \cN)
 \cup \supp \suExt^{1}_{\cO_Y}(\cM, \cN).
$$
If $N$ is also perfect over $Q$, then there is an
equality
\[V_Q^{\vf}(M,N) = \supp \cH^{2i} \HomMF( \bE_M, \bE_N) \cup \supp
\cH^{2i+1} \HomMF( \bE_M, \bE_N)\]
for any $i \in \Z$.
This follows from \ref{stable_ext_hommf} and the definition of
$\HomMF( \bE_M, \bE_N)$ as a twisted periodic complex which shows that
$\cH^{i+2} \HomMF( \bE_M, \bE_N) \cong (\cH^{i} \HomMF( \bE_M, \bE_N)
)(1)$ for all $i \in \Z$.
\end{rem}

\begin{rem}
 In fact, there is a containment
\[\supp
 \suExt^{0}_{\cO_Y}(\cM, \cN) \supseteq \supp
 \suExt^{1}_{\cO_Y}(\cM, \cN).\] Indeed, for a local hypersurface ring
 $T$, and a finite $T$-module $M$, it follows from \cite[4.2]{AvBu00}
 that $M$ has finite projective dimension if and only if $\sExt^0_T(M,
 M ) = 0$ which implies that $\sExt^1_T(M,M) = 0$.
\end{rem}

\section{Proof of Theorem~\ref{thm:Eisenbud}}
\label{operator_proof_section}

In this section we prove Theorem
\ref{thm:Eisenbud}. The first step is to associate to a matrix factorization $\bE$ an explicit complex of finitely generated
projective modules that represents $\Psi(\bE)$ in $\Dsing(R)$, where
$\Psi$ is the equivalence of Corollary \ref{orlovcor}. We remain in
the context and under the assumptions of \S \ref{assumptions}.

For any scheme $X$, we write $\cC(X)$ for the category of arbitrary
complexes of quasi-coherent sheaves and $\cC^b(X)$ for the full
subcategory consisting of complexes with bounded and coherent cohomology. If $X = \Spec R$, we write
  $\cC(R)$ and $\cC^b(R)$.

\begin{defn}\label{def:flat}
Set $\RQ^\sharp$ to be the additive functor from $MF(\ps Q, \cO(1), W)_\exact$ to
  $\cC^b(\ps R)$ that sends an object $\bE = (\cE_1 \to \cE_0 \to \cE_1(1))$ to
$$
\RQ^\sharp \bE:= \left( 
\cdots \to \RQ^*\cE_0(-2) \to \RQ^*\cE_1(-1) \to
\RQ^*\cE_0(-1) \to 0\right),
$$
where $\RQ^*\cE_0(-1)$ is located in cohomological degree
$c-1$. The action on morphisms is the obvious one. Recall that $\RQ: \ps R
\into \ps Q$ is the canonical inclusion.
\end{defn}
 
\begin{rem}   \label{RQ_bdd_cohom}
The complex $\RQ^\sharp \bE$ is exact everywhere except on
  the right (i.e., except in
  degree $c-1$). Indeed by \cite[5.2.1]{BW11a}, $\YQ^*
  \bE$ is an exact complex of locally free sheaves, where $\YQ: Y
  \into \ps Q$ is the canonical inclusion and $\YQ^* \bE$ is defined in \eqref{mf_pullback_defn}. The sheaf $\cM =
  \cH^0( \YQ^* \bE^{\leq 0} ) = \coker \partial^0_{\YQ^* \bE}$ has an obvious right resolution
  by locally free sheaves. Thus since $\RY$ has finite flat dimension,
  where $\RY: \ps R \into Y$ is the canonical inclusion, $\L
  \RY^* \cM \cong \RY^* \cM$, i.e.\ $\L
  \RY^* \cM$ has cohomology only
  in a single degree. Since $\YQ^* \bE^{\leq 0}$ is a locally free
  resolution of $\cM$, this translates to the fact that $\RY^* \YQ^*
  \bE^{\leq 0} \cong \RQ^* \bE^{\leq 0}$
  only has cohomology in a single degree. The claim follows
  immediately.
\end{rem}

Recall that $\RY: \ps R \into
Y$ is the canonical inclusion, and that $\RY^\flat$ is the right
adjoint to $\R \RY_*$. By \cite[III.1.5, III.7.3]{MR0222093}, the functor $\RY^\flat$ sends a complex
$\cF$ to $\L\RY^*(\cF)(-1)[-c+1]$.
By \cite[5.2.1]{BW11a}
$$ \cdots
\to \YQ^* \cE_1(-1) \to \YQ^* \cE_0(-1) \to 0
$$
is a locally free resolution of $\coker(\bE)(-1)$. Thus there is a
functorial isomorphism in $\Db( \ps R)$:
\[
\L\RY^*(\coker(\bE)(-1))[-c+1] \cong
\RQ^\sharp \bE,
\]
since $\YQ \circ \RY= \RQ$.
We thus have a functorial isomorphism
\begin{equation}
  \label{eq:1752}
  \RQ^\sharp \bE \xra{\cong} \beta^\flat \coker \bE 
\end{equation}
in the category $\Db(\ps R)$.

Let $\cU = (U_1, \dots, U_c)$ denote the standard affine cover of $\ps R$
and let $\cP$ be a complex of quasi-coherent sheaves on $\ps R$. As
in Section \ref{setup}, $\Gamma(\cU, \cP)$ is
 the total complex associated to the bicomplex given by applying the
 Cech construction with respect to the cover $\cU$ degree-wise to $\cP$. We view this as a functor $\Gamma(
 \cU, - ): \cC^b(\ps R) \to \cC^b(R)$. There is a natural
 isomorphism
\begin{equation}
\label{eq:2246}
\Gamma( \cU, \cP ) \xra{\cong} \R \pi_* \cP \in \Db(\ps
R). 
\end{equation}
Combining \eqref{eq:1752} and \eqref{eq:2246} we see that for $\bE$ a
matrix factorization, we have a natural isomorphism:
\[ \R\pi_* \RY^\flat(\coker(\bE)) \cong \Gamma( \cU, \RQ^\sharp \bE).\]
Since $\R\pi_* \RY^\flat(\coker(-))$ induces the equivalence $\Psi: [MF] \to \rDsg(Q \onto R)$ of
Corollary \ref{orlovcor}, we have
proven the following:

\begin{prop}\label{main_equiv_explicit}
The functor $\Gamma( \cU, \RQ^\sharp -)$ induces a functor $[MF(\bP_Q^{c-1},
\cO(1), W)] \to \rDsg(Q \onto R)$, which we also write as $\Gamma( \cU,
\RQ^\sharp -)$, and there is an isomorphism of triangulated functors
\[ \Psi( - ) \xra{\cong} \Gamma( \cU, \RQ^\sharp -): [MF(\bP_Q^{c-1},
\cO(1), W)] \to \rDsg(Q \onto R), \]
where $\Psi$ is the equivalence of Corollary~\ref{orlovcor}.
\end{prop}

As the sheaves in the complex $\Gamma( \cU,
\RQ^\sharp \bE)$ are not coherent, we
look for a complex of coherent sheaves that represents it in $\rDsg(Q
\onto R)$.

\begin{defn} \label{def:F} Define $F: \cC^b(\ps R) \to
  \cC(R)$ to be the functor given by applying
$$
\rH^{c-1}(\ps R, -)[-c+1]
$$ 
degree-wise. In particular
$$
\begin{aligned}
F(\RQ^\sharp \bE) :=
\big(
\cdots & \to \rH^{c-1}(\ps R, \RQ^*\cE_0(-2)) \to
\rH^{c-1}(\ps R, \RQ^* \cE_1(-1))  \\
 & \to
\rH^{c-1}(\ps R, \RQ^* \cE_0(-1)) \to 0 \big),
\end{aligned}
$$
with $\rH^{c-1}(\ps R, \cE_0(-1))$ located in cohomological degree
$2c-2$.  
\end{defn}

If $\cE$ is any coherent sheaf on $\ps R$, there is a natural map of complexes
\begin{equation}\label{1830_eq}
\Gamma(\cU, \cE) \to \rH^{c-1}(
\Gamma(\cU, \cE) )[-c+1]
\end{equation}
since $(\Gamma(\cU, \cE))^i = 0$ for $i > c -1$.
Since $\rH^{i}(\Gamma( \cU,
\cE)) \cong \rH^i( \ps R, \cE)$, the map \eqref{1830_eq} is a quasi-isomorphism if
and only if $\rH^i(\ps R, \cE) = 0$ for all $i \ne c-1$.
Applying \eqref{1830_eq} degree-wise to a chain complex $\cP$ yields a map of
bicomplexes and
taking the total complex gives a map
$\eta_\cP: \Gamma(\cU, \cP) \to F( \cP ).$
It is natural in $\cP$ and so gives a natural transformation
\begin{equation} \label{etadef}
\eta: \Gamma(\cU, -) \to F(-).
\end{equation}
{}From this construction we see that if $\rH^i(\ps R, \cP^j) = 0$ for
all $i \ne c-1$ and all $j$ then $\eta_\cP$ is a
quasi-isomorphism. Also note that
for all $\cP$ there is a commutative diagram
\begin{equation}
\xymatrix{ \Gamma(
  \cU, \cP[1] ) \ar[d]_\cong \ar[r]^{\eta_{\cP[1]}} & \ar[d]^\cong F( \cP[1] ) \\
\Gamma( \cU, \cP )[1] \ar[r]_{\eta_{\cP}[1]} & 
F(\cP)[1]}
\label{eq:graded_func} 
\end{equation}
where $[1]$ is the shift functor on the
respective categories of complexes.

For a complex $\cP$, recall that  $\cP^{\leq m}$ is the brutal
truncation in cohomological degree $m$.
There are natural maps $\cP \to \cP^{\leq m}$ and $\cP^{\leq m}
\to \cP^{\leq n}$ for $m \geq n$.

\begin{defn} 
\label{def_suff_small}
An
integer $m$ is {\em sufficiently small} for an object
$\bE = (\cE_1 \to \cE_0 \to \cE_1(1))$ of $MF(\ps Q, \cO(1), W)$ if every coherent sheaf $\cE$ appearing in the
complex $\RQ^\sharp(\bE)^{\leq m} = (\cdots \to \RQ^* \cE_1(-1) \to
\RQ^* \cE_0(-1) \to 0)^{\leq m}$ satisfies 
\begin{enumerate}
\item   $\rH^i(\ps R, \cE) = 0$ for all $i \ne c-1$ and 
\item $\rH^{c-1}(\ps R, \cE)$ is a 
projective $R$-module.
\end{enumerate}
\end{defn}
The name of this condition is justified by:

\begin{lem} \label{lem:good}
Let $\bE$ be an object of $MF(\ps Q, \cO(1), W)$. There is
an integer $m_0$, which
depends on $\bE$, such
that all $m \leq m_0$ are sufficiently small for $\bE$. Moreover, if
$m$ is sufficiently small for $\bE$, then $F\left ( (\RQ^\sharp
  \bE)^{\leq m} \right )$ 
is a right bounded complex of finitely generated projective
$R$-modules and the map
$$
\eta: \Gamma\left (\cU, (\RQ^\sharp \bE)^{\leq m} \right )  \map{\sim} F\left ( (\RQ^\sharp
  \bE)^{\leq m} \right )
$$ 
is a quasi-isomorphism.
\end{lem}

\begin{proof}
Consider the complex $\RQ^\sharp \bE$. By construction, if for some
$i$, the sheaf $(\RQ^\sharp \bE)^{i}$ is nonzero,
then 
\begin{equation}
\label{E1527}
(\RQ^\sharp \bE)^{i-2}
\cong (\RQ^\sharp \bE)^i(-1).
\end{equation}
Let $(-)^\vee$ denote
$\uHom_{\ps R}({-},{\cO_{\ps R}})$. It follows from 
Serre Vanishing and the isomorphism \eqref{E1527} that
for $m \ll 0$,
every coherent sheaf $\cE$ in
$\RQ^\sharp \bE^{\leq m}$ satisfies 
\[\rH^i(\ps R, \cE(-c)^\vee) = 0 \text{ for all } i \ne 0.\]
By \cite[7.9.10]{MR0163911}, $\rH^0(\ps R, \cE(-c)^\vee)$ is thus a finitely
generated projective $R$-module for each such $\cE$.
By a corollary to Serre-Grothendieck
Duality, see \cite[III.5.2]{MR0222093}, we get
$$
\begin{aligned}
\rHom_R(\rH^i(\ps R, \cE(-c)^\vee), R)
& \cong \rExt^{c-1-i}_{\cO_{\ps R}}(\cE(-c)^\vee, \cO(-c)) \\
& \cong \rExt^{c-1-i}_{\cO_{\ps R}}(\cO_X, \cE) \\
&\cong \rH^{c-1-i}(\ps R, \cE)
\end{aligned}
$$
for all $\cE$ in $\RQ^\sharp \bE^{\leq m}$ and all $i$. The first two
assertions follow.

The last part was noted below the definition of $\eta$.
\end{proof}

\begin{prop} \label{prop:F}
For every object $\bE$ of $MF(\ps Q, \cO(1), W)$ and integer $m$
that is sufficiently small for $\bE$, there is an isomorphism
$$
\zeta_{\bE}^m: \Gamma(\cU,\RQ^\sharp \bE) \map{\cong} F\left ( (\RQ^\sharp \bE)^{\leq
    m} \right ) \text{
  in } \Dsg(R)
$$
such that the following naturality condition holds: given a strict map $g: \bE \to
\bF$ and an integer $n \leq m$ that is sufficiently small for $\bF$, the diagram
$$
\xymatrix{
\Gamma(\cU,\RQ^\sharp \bE) \ar[dd]^{\Psi(g)}  \ar[r]^(.4){\cong}_(.4){\zeta_{\bE}^m}   &
F\left ( (\RQ^\sharp \bE)^{\leq
    m} \right )
\ar[d]^{F\left ( (\RQ^\sharp g)^{\leq
    m} \right )} \\
 & F\left ( (\RQ^\sharp \bF)^{\leq
    m} \right ) \ar[d]^{canonical} \\
\Gamma(\cU,\RQ^\sharp \bF) \ar[r]^(.4){\cong}_(.4){\zeta_{\bF}^n}      &   F\left ( (\RQ^\sharp \bE)^{\leq
    n} \right )   \\
}
$$
commutes in $\Dsg(R)$.
\end{prop}

\begin{proof}
Given such a pair $(\bE, m)$, define $\zeta_{\bE}^m$ 
to be the map in $\Dsg(R)$ represented by the composition of 
\begin{equation}
\Gamma(\cU,\RQ^\sharp \bE) \map{\eta} F(\RQ^\sharp \bE)
\to F\left ( (\RQ^\sharp \bE)^{\leq
    m} \right ),\label{eq:psibe-xrac-gamm}
\end{equation}
where
$F(\RQ^\sharp \bE)
\to F\left ( (\RQ^\sharp \bE)^{\leq
    m} \right )$ is induced by the canonical  map
$\RQ^\sharp \bE \to (\RQ^\sharp \bE)^{\leq m}$.

The naturality of $\eta$ shows that the diagram
\begin{equation} \label{L624}\xymatrix{\Gamma(\cU,\RQ^\sharp \bE) \ar[d] \ar[r]_\eta & F(\RQ^\sharp \bE) \ar[d] \\
\Gamma\left (\cU, (\RQ^\sharp \bE)^{\leq  m}\right ) \ar[r]_\eta &
F\left ( ( \RQ^\sharp
\bE)^{\leq  m} \right )}
\end{equation}
commutes, where the vertical maps are induced by $\RQ^\sharp \bE \to \RQ^\sharp
\bE^{\leq m}$.
Thus to show that \eqref{eq:psibe-xrac-gamm} is an
isomorphism it is enough to show that the left vertical map and
the lower
horizontal map in \eqref{L624} are isomorphisms.

The lower horizontal map  of \eqref{L624} is an isomorphism in $\Dsg(R)$ by Lemma \ref{lem:good}. The cone of $\RQ^\sharp(\bE) 
\to \RQ^\sharp(\bE)^{\leq m}$ is a bounded complex of locally free
coherent sheaves, i.e.\ is a perfect complex, and $\R\pi_* = \Gamma(\cU, -) $ maps
perfect complexes to perfect complexes by \cite[III.4.8.1]{MR0354655}. Thus the
left vertical map of \eqref{L624} is an isomorphism in $\Dsg(R)$.

The naturality assertion is evident from the construction.
\end{proof}

\begin{rem}
By \ref{eq:graded_func} and the construction of $\zeta$ above,
we see that there is a commutative diagram
\begin{equation} \xymatrix{
\Gamma\left ( \cU, \RQ^\sharp
  (\bE[2]) \right ) \ar[r]^\cong \ar[d]_{\cong}& F\left ( (\RQ^\sharp \bE[2])^{\leq m}\right) \ar[d]^{\cong} \\
\Gamma( \cU, \RQ^\sharp
  \bE )[2] \ar[r]^\cong & F\left ( (\RQ^\sharp \bE)^{\leq
      m}\right)[2].}
\label{graded_funct}
\end{equation}
\end{rem}

\subsection{Recollection of Eisenbud operators} \label{Eops}

We recall from \cite[\S 1]{Ei80} the construction of the Eisenbud operators
for the commutative ring $R = Q/(\vf)$, where $\vf = (f_1, \dots,
f_c)$ is a  $Q$-regular sequence. Note that the ring $Q$ need not be regular for this construction.

Let $F$ be a complex of projective $R$-modules such that
there exists a sequence of maps of projective $Q$-modules 
\[
\widetilde{F} = \cdots \to \widetilde{F}^{p-1} \xra{\tilde{\partial}^{p-1}} \widetilde{F}^p \xra{\tilde{\partial}^p}
\widetilde{F}^{p+1} \to \cdots
\]
 with $F
\cong \widetilde{F} \otimes_Q R$. The sequence of maps $\widetilde{F}$ is {\em not}
required to be a chain complex. Since $\widetilde{F} \otimes_Q R \cong F$, for each $i = 1, \dots, c$ and
all $n$, there exist (non-unique) maps
$\tilde{t}_i^n: \widetilde{F}^{n} \to \widetilde{F}^{n+2}$ 
such that
\begin{equation}
\label{eis_ops_cond}
\tilde{\partial}^{n+1} \tilde{\partial}^{n} = \sum_{i = 1}^c
\tilde{t}_i^n f_i.
\end{equation}
Set $t_i^n =  \tilde{t}_i^n \otimes_Q R:
F^{n} \to F^{n+2}$. Then the following properties hold \cite[Section 1]{Ei80}:
\begin{enumerate}
\item for each
$i$, the maps $t_i^n$ 
assemble to
give morphisms of chain complexes,
$t_i: F \to F[2]$;
\item the chain map $t_i$ is independent of the choice of
$\widetilde{F}$ and the $\tilde{t}_i^n$'s, up to homotopy;
\item the $t_i$'s commute, up to homotopy;
\item the $t_i$'s are
natural, up to homotopy, in the argument $F$.
\end{enumerate}
It is assumed in \cite{Ei80} that $F$ is a complex of free
$R$-modules, in which case a lifting $\widetilde{F}$ always
exists: viewing the
differentials of $F$ as matrices, simply
lift each element in each matrix to an element of $Q$. However, if $F$ is only assumed to be a complex of projective
modules for which such a lifting exists, then the proofs
in \cite{Ei80} go through unchanged.

For any object $M$ in $\Db(R)$ we may choose a complex
of projective $R$-modules $F$, an isomorphism  $\alpha: F \map{\cong}
M \in \Db(R)$, and a lifting $\widetilde{F}$ of $F$ as above. For
example, one could take $F$ to be a bounded above complex of free
modules that maps via a quasi-isomorphism to $M$.
Define $t_i^M: M \to M[2]$ in $\Db(R)$
to be the composition of
$$
M \map{\alpha^{-1}} F \map{t_i} F[2] \map{\alpha[2]} M[2],
$$
where the map $t_i$ is induced from a choice of maps $\tilde{t_i^n}$
satisfying \eqref{eis_ops_cond}.

The facts listed above imply that $t_i^M$ does not depend on the
choice of $F$ or $t_i$, that $t_i^M \circ t_j^M = t_j^M \circ
t_i^M$ for all $i,j$ and, if $g: M \to N$ is a morphism in $\Db(R)$,
then
$$
\xymatrix{
M \ar[r]^{t_i^M}  \ar[d]^g & M[2] \ar[d]^{g[2]} \\
N \ar[r]^{t_i^N} & N[2] \\
}
$$
commutes. In other words, the collection of maps $t_i^M$ form a
family of pairwise commuting natural transformation from the identity functor to the functor
$(-)[2]$.
We call these the {\em Eisenbud operators on $\Db(R)$} given by
$(Q, \vf)$. 

It is clear that $t_i^M$ descends along the canonical functor
$\Db(R) \to \Dsg(R)$ to give natural
transformations on $\Dsg(R)$, and we will also write these
induced transformations as $t_i^M$. Moreover, these induced
transformations, in turn, restrict to give
natural transformations of endo-functors on the full, triangulated subcategory $\rDsg(Q
\onto R)$ of $\Dsg(R)$.

\begin{proof}[Proof of Theorem \ref{thm:Eisenbud}]
\label{proof_of_main_thm_sec3}
Let $\bE$ be an object of $[MF(\ps Q, \cO(1), W)]$. Let $m$ be an
integer which is sufficiently small for $\bE[2]$; by definition it
will also be sufficiently small for $\bE$. By \ref{lem:good} and \ref{prop:F}, 
\mbox{$ \zeta^m: \Gamma( \cU, \RQ^\sharp
  \bE ) \to F\left ( (\RQ^\sharp \bE)^{\leq m}\right)$} is an
isomorphism in $\Dsg(R)$ and $F\left ( (\RQ^\sharp \bE)^{\leq m}\right)$ is a complex of finitely generated
projective $R$-modules. Using the
  isomorphism of functors $\Psi( - ) \xra{\cong} \Gamma( \cU, \RQ^\sharp -)$
  of Proposition \ref{main_equiv_explicit}, we have a commutative
  diagram:
\[ \xymatrix{ \Psi( \bE ) \ar[r]^\cong \ar[d]_{\Psi(T_i)}& \Gamma( \cU, \RQ^\sharp
  \bE ) \ar[r]^\cong \ar[d]_{\Gamma(\cU, T_i)}& F\left ( (\RQ^\sharp
    \bE)^{\leq m}\right) \ar[d]^{F( \RQ^\sharp(T_i)^{\leq m})} \\
\Psi( \bE[2] ) \ar[r]^\cong \ar[d]_{\cong}& \Gamma\left ( \cU, \RQ^\sharp
  (\bE[2]) \right ) \ar[r]^\cong \ar[d]_{\cong}& F\left ( (\RQ^\sharp \bE[2])^{\leq m}\right) \ar[d]^{\cong} \\
\Psi( \bE )[2] \ar[r]^\cong & \Gamma( \cU, \RQ^\sharp
  \bE )[2] \ar[r]^\cong & F\left ( (\RQ^\sharp \bE)^{\leq m}\right)[2]}
\]
Indeed, the top left square commutes since it is given by a natural
transformations applied to the map $T_i$. The top right square
commutes by Proposition \ref{prop:F}. The lower left square commutes
since $\Psi( - ) \xra{\cong} \Gamma( \cU, \RQ^\sharp -)$ is an
isomorphism of triangulated functors, and the lower right square commutes by
\eqref{graded_funct}.

Since $F\left ( (\RQ^\sharp \bE)^{\leq m}\right)$ is a complex of projective
$R$-modules, we may calculate Eisenbud operators using it. Let us call
the right hand vertical map in the above diagram $\beta_i: F\left (
  (\RQ^\sharp \bE)^{\leq m}\right) \to F\left ( (\RQ^\sharp \bE)^{\leq
    m}\right)[2]$. To prove Theorem~\ref{thm:Eisenbud} it is enough to find
a lifting of $F\left ( (\RQ^\sharp \bE)^{\leq m}\right)$ to a complex
of projective $Q$-modules and degree $-2$ endomorphisms $\tilde{t}_i$
of the lifting such that
\ref{eis_ops_cond} holds, and such that $\tilde{t_i} \otimes_Q R \cong \beta_i$.

To lift $F\left ( (\RQ^\sharp \bE)^{\leq m}\right)$
to a sequence of projective $Q$-modules consider the sequence of maps of $Q$-modules defined in analogy with Definition
\ref{def:F}:
$$
\begin{aligned}
\tF(\bE) :=
& \big(
\cdots \to \rH^{c-1}(\ps Q, \cE_0(-2)) \to
\rH^{c-1}(\ps Q, \cE_1(-1)) \\
& \to
\rH^{c-1}(\ps Q, \cE_0(-1)) \to 0\big),
\end{aligned}
$$
with $\rH^{c-1}(\ps Q, \cE_0(-q))$ located in cohomological degree
$2c-2$. 
Just as for $F\left ( (\RQ^\sharp \bE)^{\leq m}\right)$, if 
$m \ll 0$, each coherent sheaf in $\tF( \bE)^{\leq m}$ 
is a finitely
generated projective $Q$-module (but this is not a chain complex). We claim there is an isomorphism of chain complexes 
$$
\tF( \bE)^{\leq m}
\otimes_Q R \cong F\left ( (\RQ^\sharp \bE)^{\leq m}\right)
$$ 
for all $m \ll 0$. Indeed, if $\cE$ is any coherent
sheaf on $\ps Q$, then by \cite[II.5.12]{MR0222093} we have the isomorphism $\R \pi_* \L \RQ^* \cE \cong R \otimes_Q^{\L}
\R p_* \cE$ in $\Db(R)$, where $p: \ps Q \to \Spec Q$ is the canonical map. Moreover,
if $\cE$ is locally free and $\rH^i(\ps Q, \cE) = 0$ for all $i \ne c-1$, then
this gives an isomorphism of $R$-modules
$$
\rH^{c-1}(\ps R, \RQ^* \cE) \cong
\rH^{c-1}(\ps Q, \cE) \otimes_Q R,
$$
and the claim follows.

Thus $\tF( \bE)^{\leq m}$ is a lifting; to see it is the lifting we
sought, note that $\beta_i$ is given in each degree by \[\rH^{c-1}(\ps R, \cE_j(k) )
\xra{\rH^{c-1}(\ps R, T_i)} \rH^{c-1}(\ps R, \cE_j(k+1)).\]
The composition of two successive maps in $\tF( \bE)^{\leq m}$ is the map on
sheaf cohomology induced by multiplication by  $\sum_{i=1}^c T_i f_i$.
Thus, letting $\tilde{t}_i = \rH^{c-1}(\ps Q, T_i)$ we see that
$\beta_i = \tilde{t}_i \otimes_Q R$ and the theorem follows.
\end{proof}

\section{Projective resolutions}
\label{sec_proj_res}
We continue to work in the context and under the assumptions of
\S \ref{assumptions}, with the added assumption that $Q$ is Gorenstein,
i.e.\ $Q$ has finite injective dimension over itself. Let $M$ be a finitely generated $R$-module that has
finite projective dimension over $Q$.
In this section we construct a projective
resolution over $R$ of a high
syzygy of $M$ using the equivalence $\Psi$ of Corollary
\ref{orlovcor} and the explicit representative of $\Psi(\bE)$ given in
Section \ref{operator_proof_section}.

Recall that a coherent
sheaf $\cF$ on $\ps R$ is
  \emph{$m$-regular} for an integer $m$ if \[\rH^i(\ps R, \cF(m-i) ) = 0 \text{ for all } i >
  0.\] The
  \emph{regularity} of $\cF$ is the smallest $m$ such that $\cF$ is
  $m$-regular; see e.g.\ \cite [Lecture 14]{MR0209285} and
  \cite[Section 8]{MR0338129}. By Serre's Vanishing Theorem, every
coherent sheaf on $\ps R$ has finite regularity.

\begin{defn}
For a matrix factorization $\bE = (\cE_1 \to
 \cE_0 \to \cE_1(1))$, we set 
\[\alpha( \bE ) = \max \{ \text{regularity } (\RQ^*\cE_0)^\vee - 1, \text{regularity
} (\RQ^* \cE_1)^\vee - 1\},\]
where $\RQ: \ps R \to \ps Q$ is the canonical inclusion, and $(-)^\vee =
\uHom( -,
\cO )$.
\end{defn}

\begin{thm}
\label{proj_res_thm}
 Let $R = Q/(f_1, \ldots, f_c)$, where $Q$ is a Gorenstein ring of
 finite Krull dimension and $f_1, \ldots, f_c$ is a $Q$-regular
 sequence. Let $M$ be a finitely generated $R$-module that has finite projective
 dimension over $Q$ and let $\bE = (\cE_1 \to \cE_0 \to \cE_1(1))$ be an
 object of $MF(\ps Q, \cO(1), W)$ such that $\Psi(\bE) \cong M$, where
 $\Psi$ is the equivalence of \ref{orlovcor}. Set
\[ n_{\bE} = 2\alpha(\bE) + e - 1,\] where $e$ is the is the Krull
  dimension of $Q$. The complex 
\[F(\RQ^\sharp \bE)^{\leq -n_\bE} = \cdots \to \rH^{c-1}(\ps R, \RQ^* \cE_1(m)) \to \rH^{c-1}(\ps R, \RQ^*
\cE_0(m)) \to \rH^{c-1}(\ps R, \RQ^* \cE_1(m+1)) \to \cdots
\]
 is a projective resolution of a $n_\bE$-th syzygy
of $M$, where $\RQ^\sharp \bE$ is defined
in \ref{def:flat} and $F(-)$ is defined in \ref{def:F}.
\end{thm}

\begin{rem}
 \label{eis_conj} On this resolution one may choose the Eisenbud
 operators
\[ t_i : \rH^{c-1}(\ps R, \RQ^* \cE_i(n)) \to \rH^{c-1}(\ps R, \RQ^*
\cE_i(n+1)) \]
to be multiplication by $T_i \in \Gamma( \ps Q,
 \cO_{\ps Q}(1))$. In particular these operators commute. 
\end{rem}

We need several preliminary results for the proof of Theorem \ref{proj_res_thm}.

\begin{lem}
\label{suff_small_terms_of_regularity}
The integer $-2\alpha(\bE)  - c + 1$ is
sufficiently small for the matrix factorization $\bE$, in the sense of Definition \ref{def_suff_small}.
\end{lem}

\begin{proof}
By \cite[8.1.3]{MR0338129}, if a sheaf $\cF$ on $\ps R$ is $k$-regular
then it is also $(k+1)$-regular. Thus the sheaves $(\RQ^*\cE_0)^\vee$ and
$(\RQ^*\cE_1)^\vee$ are $k$-regular for all $k > \alpha(\bE)$. In
particular 
\[\rH^i( \ps R, (\RQ^*\cE_0)^\vee(k) ) = 0 = \rH^i(
\ps R, (\RQ^*\cE_1)^\vee(k) ) \text{ for all } i > 0 \text{ and all } k \geq
\alpha(\bE).\] As in the proof of \ref{lem:good}, this implies that for
$j = 0,1$ and all $k \geq \alpha(\bE)$:
\begin{enumerate}
\item the $R$-module $\rH^0( \ps R, (\RQ^*\cE_j)^\vee(k))$ is projective;
\item there is an isomorphism \[  \Hom R {\rH^0( \ps R,
  (\RQ^*\cE_j)^\vee(k))} R\cong \rH^{c-1}( \ps R,
  (\RQ^*\cE_j)(-k - c)),\] and
  so in particular $\rH^{c-1}( \ps R,
  (\RQ^*\cE_j)(-k - c))$ is projective;
\item $\rH^i( \ps R,
  (\RQ^*\cE_j)(-k - c)) = 0$ for all $i < c - 1$.
\end{enumerate}
Now consider the complex $$
\RQ^\sharp \bE := \left( 
\cdots \to \RQ^*\cE_0(-2) \to \RQ^*\cE_1(-1) \to
\RQ^*\cE_0(-1) \to 0\right)
$$
where $\RQ^*\cE_0(-n)$ is in cohomological degree $c - 2n + 1$.
We have
$$
\begin{aligned}  (\RQ^\sharp \bE)^{\leq -2\alpha(\bE) - c +1} := ( 
\cdots &\to \RQ^*\cE_0(-(\alpha(\bE)+1) - c) \to \\
\RQ^*\cE_1(-\alpha(\bE)-c) &\to
\RQ^*\cE_0(-\alpha(\bE) - c) \to 0) \end{aligned} 
$$
and by the above each sheaf appearing in this complex satisfies
the two conditions of Definition \ref{def_suff_small}.
\end{proof}

\begin{prop}
\label{bdd_cohom}
  The cohomology of the complex $F(\RQ^\sharp \bE)^{\leq
    -2\alpha(\bE)}$ is concentrated in
  degrees $-2 \alpha(\bE) - c + 1 \leq i \leq -2\alpha(\bE)$.
\end{prop}

\begin{proof} {}From the definition of $F$ we have that $F(\RQ^\sharp \bE)^{\leq
  -2\alpha(\bE)}  = F\left ( (\RQ^\sharp \bE)^{\leq -2\alpha(\bE) - c +
    1}\right )$. Set $n = -2\alpha(\bE) - c + 1$.
By Lemma \ref{suff_small_terms_of_regularity}, $n$ is
sufficiently small for $\bE$ and thus by Lemma~\ref{lem:good} there is
an isomorphism in $\Db(R)$:
\[ 
F\left ( (\RQ^\sharp
  \bE)^{\leq n} \right )\cong \Gamma\left (\cU, (\RQ^\sharp \bE)^{\leq n} \right ).
\]
By Remark \ref{RQ_bdd_cohom} the complex of sheaves $\RQ^\sharp \bE^{\leq n}$ only
has cohomology in degree $n$. Setting $\cN = \rH^n( \RQ^\sharp \bE^{\leq
  n} )$ gives an isomorphism in $\Db(Y)$:
\[ \RQ^\sharp \bE^{\leq n} \cong \cN[-n]. \]
Since the functor $\Gamma( \cU, - )$
represents $\R \pi_*$, we have the following isomorphisms in $\Db(R)$:
\[ \Gamma\left (\cU, \RQ^\sharp (\bE)^{\leq n}\right ) \cong  \R \pi_*
\left (\RQ^\sharp (\bE)^{\leq n}\right )  \cong \R \pi_* (\cN[-n] ) \cong
\R \pi_* (\cN)[-n]. \]
Since $\R \pi_*$ has cohomological dimension $c-1$, we see that $\R
\pi_* (\cN)$ has cohomology in degrees at most $0 \leq i \leq c
-1$. Combining the above gives
\begin{align*} F(\RQ^\sharp \bE)^{\leq
  -2\alpha(\bE)}  &=  F\left (\RQ^\sharp (\bE)^{\leq n}\right ) \\
&\cong \Gamma\left (\cU, \RQ^\sharp (\bE)^{\leq n}\right ) \\ & \cong \R \pi_* (\cN )[-n]\end{align*} has
cohomology in degrees at most $n \leq i \leq n + c - 1$.
\end{proof}

\begin{proof}[Proof of Theorem \ref{proj_res_thm}]
Let $M$ be a finitely generated $R$-module with finite projective
dimension over $Q$, and let $\bE$ a matrix factorization such that
$\Psi( \bE ) \cong M \in \Dsg(R)$. Since $-2\alpha(\bE) - c + 1 $ is
sufficiently small for $\bE$ by \ref{suff_small_terms_of_regularity},
there is, by \ref{main_equiv_explicit} and \ref{prop:F}, an isomorphism in $\Dsg(R)$ 
\[ \Psi( \bE ) \xra{\cong} F\left (
\RQ^\sharp (\bE)^{\leq -2\alpha(\bE) - c + 1}\right ) = F(\RQ^\sharp \bE)^{\leq
  -2\alpha(\bE)}. \]
Thus we have that
\[ M \cong F(\RQ^\sharp \bE)^{\leq -2\alpha(\bE)} \in \Dsg(R). \]
By Proposition \ref{bdd_cohom}, the complex $F(\RQ^\sharp \bE)^{\leq
  -2\alpha(\bE) - c + 1}$ only has cohomology in degree $ -2\alpha(\bE) - c + 1$. Moreover, since $F(\RQ^\sharp \bE)$ is a complex of projective
modules, the cone of the canonical
map \begin{equation}\label{trunc_map} F(\RQ^\sharp \bE)^{\leq
    -2\alpha(\bE)} \to F(\RQ^\sharp \bE)^{\leq
  -2\alpha(\bE) - c + 1} \end{equation}
is perfect, and hence
\eqref{trunc_map} is an isomorphism in $\Dsg(R)$. Finally, since $R$
is Cohen-Macaulay, we may truncate $F(\RQ^\sharp \bE)^{\leq
  -2\alpha(\bE) - c + 1}$ to the left $d = \dim R$ more steps so that the module \[N = \rH^{
  -2\alpha(\bE) - c - d+ 1}(F(\RQ^\sharp \bE)^{\leq
  -2\alpha(\bE) - c - d+ 1})\] is maximal Cohen-Macaulay (MCM). Note also
that $c + d$, where $c$ is the length of the regular sequence defining
$R$, is exactly the Krull dimension of $Q$, which we
called $e$. Thus $-2\alpha(\bE) - c -d + 1 = -n_\bE$ which was defined above.

We now have
an isomorphism
\[ M \cong N[-n_\bE] \in \Dsg(R). \]
Let $P = \cdots \to P^{-m} \xra{\partial^{-m}_P} P^{m+1}
 \to \cdots \to P^{-1} \xra{\partial^{-1}_P} P^0$ be a
 projective resolution of $M$ over $R$. By induction on $n$ one checks that there are isomorphisms for all $m \geq 0$:
\[ M \cong (\coker\partial^{-m}_P )[m-1] \in \Dsg(R). \]
Setting $M' = \coker \partial^{-n_{\bE}-1}_P$, which is MCM
since $n_{\bE} \geq d$, we have that
\[ M' \cong N \in \Dsg(R). \]
By \cite[4.4.1]{Bu87} this implies that $M' \cong N$ in the
stable category of MCM modules; i.e.\ there exist projective $R$-modules $Q_1, Q_2$ and an isomorphism
of $R$-modules
\[ M' \oplus Q_1 \cong N \oplus Q_2. \]
This allows us to glue the acyclic complexes
\[ \cdots \to F(\RQ^\sharp \bE)^{-n_{\bE} - 1} \to F(\RQ^\sharp \bE)^{-n_{\bE}} \to N \to 0 \]
\[ 0 \to M' \to P^{-n_\bE + 1} \to P^{-n_\bE + 2} \to \cdots \to
P^{-1} \to P^0 \to M \to 0\]
to give a projective resolution of $M$.
\end{proof}

\section{Graded matrix factorizations and a description of $\Psi^{-1}$ on objects} \label{sec:inverse}
We work in the context and under the assumptions of
\S \ref{assumptions}. Let $M$ be a finitely generated $R$-module
that has finite projective dimension
over $Q$. Our goal in this section is to give an explicit method of
constructing a matrix
factorization 
$\E_M$ such that $\Psi(\E_M) \cong M$ holds in $\Dsg(R)$, where $\Psi$
is the equivalence of \ref{orlovcor}.
We achieve this using the
data of a projective resolution of $M$ over $Q$ and a system of ``higher
homotopies'' on this resolution, as introduced by Eisenbud in \cite{Ei80}. Using the
explicit construction, we also show that $\rDsg(Q \onto R)$ may be
described using \emph{graded matrix factorizations}.

\subsection{Standard resolutions}
Let $G = 0 \to G_n \to \cdots \to G_0 \to 0$
 be a resolution of $M$ by finitely generated projective $Q$-modules.
 
For an element $J = (a_1, \ldots, a_c) \in \N^c$, we set $|J| =
  \sum_1^c a_i$. We write the element $(0, \ldots, 1, \ldots, 0)$,
  where the $1$ in the $i$-th position, as $i$.
By \cite[7.1]{Ei80} there exists a
  family of endomorphisms of $G$
  \[\boldsymbol{\sigma} = \brc{ \sigma^J}{J \in \N^c},\]
  where $\sigma^J \in \boldsymbol{\sigma}$ has degree $2|J| - 1$
  (i.e.\ the components are $\sigma^J_j: G_j \to G_{j + 2|J| - 1}$),
  that satisfy the following equations:
  \begin{align*}
    \sigma^0 &= \partial_G \\
    \sigma^0 \sigma^i + \sigma^i \sigma^0 &= f_i 1_G \\
    \sum_{J' + J'' = J} \sigma^{J'} \sigma^{J''} &= 0 \text{ for all }
    J \in \N^c \text{ with } |J| \geq 2.
  \end{align*}
  Such a family $\boldsymbol{\sigma}$ is called a \emph{system of
    higher homotopies}.

  \begin{rem}
   The result above is stated for complexes of free modules in
   \cite{Ei80} but the proof works under the weaker assumption that
   the modules are projective.
  \end{rem}

\begin{ex}
Let $K = K \langle e_1, \ldots, e_c \, | \, \partial( e_i) = f_i\rangle$ be the Koszul
complex resolving $R$ over $Q$. Then $K$ has the structure of a
differential graded (DG)-algebra
with multiplication given by the alternating product.
For an $R$-module $M$ with finite projective dimension over $Q$,
we may find a finite projective resolution $G$ of $M$
over $Q$ such  that $G$ is a DG-module over $K$
(and such that the $Q$-module structure coincides with the restriction along
the canonical map $Q \to K$). See \cite[Section 2]{MR1774757} for an
explicit construction of such a $G$.
Setting
$\sigma^0 = \partial^G$, $\sigma^i$ to be multiplication by $e_i$ for $i = 1, \ldots, c$,
and $\sigma^J = 0$ whenever  $|J| \geq 2$
 gives a system of higher homotopies on $G$. While such a system of higher homotopies is simpler than the general
type, not every $Q$-projective resolution of $M$ has the structure of a DG-module over $K$; see \cite{MR601460}.
\end{ex}

Define $D = \bigoplus_{j \leq 0} D^j = \bigoplus_{j \leq 0} \rHom_R(R[\chi_1, \dots,
\chi_c]_{-j},R)$ to be the graded dual of $R[\chi_1, \dots, \chi_c]$,
where $|\chi_i| = 2$, which we regard as a graded $R[\chi_1, \dots, \chi_c]$-module in the obvious way.
(In fact $D$ is a divided power algebra,
but we ignore the multiplication rule and regard it only as a graded
$R[\chi_1, \dots, \chi_c]$-module.) 
Consider the graded $R$-module
$$
\overline{G} \otimes_R D,
$$
where $G$ is the projective $Q$-resolution of $M$ above, and $\overline{G} = G \otimes_Q R$. 
 Note that $\overline{G}_i \otimes_R D^{2j}$
lies in cohomological degree $2j-i$.
Let $\boldsymbol{\sigma}$ be a system of higher homotopies on $G$, and
define 
$\partial = \sum_{J \in \N^c} \sigma^J \otimes \chi^J$, where $\chi^J =
  \chi_1^{a_1}\ldots \chi_c^{a_c}$ for $J = (a_1, \ldots, a_c) \in
  \N^c$, a degree -1 endomorphism of $\overline{G} \otimes_R D$.
When $R$ is local, \cite[Theorem 7.2]{Ei80} shows that this complex 
is an $R$-free resolution of $M$. By localizing, we see that, in general,
$\overline{G} \otimes_R D$ is a projective $R$-resolution of
$M$. This is a \emph{standard resolution of $M$} which we write as
$G \{ \boldsymbol{\sigma} \}$.

\subsection{Graded matrix factorizations}
Recall that  $S$ is
  the graded $Q$-algebra $S := Q[T_1, \dots,
  T_c]$, with $|T_i| = 1$, and $W
  = f_1T_1 + \cdots f_c T_c \in S$.
\begin{defn}
 A \emph{graded matrix factorization} of $W$ is a pair of graded
  free $S$-modules $E_1, E_0$ and maps
  \[ E_1 \xra{g_1} E_0 \xra{g_0(1)} E_1(1) \] such that $g_0 \circ
  g_1$ and $g_1(1) \circ g_0$ are multiplication by $W$. Exactly
  analagously to $MF(X, \cL, W)$ there is a category $MF^{gr}(S,W)$
  whose morphisms are pairs of maps that cause the obvious squares to
  commute. One defines homotopies between these maps and the homotopy
  category $[MF^{gr}(S,W)]$ analagously to $[MF(X, \cL, W)]_\naive$. By \cite[3.4]{MR2641200}, $[MF^{gr}( S, W)]$ is triangulated with triangles and
  shift functor defined analogously to $[MF(X,\cL,W)]_\naive$.
\end{defn}

\newcommand{\flc}{\mathcal{C}}
\newcommand{\p}{\mathfrak{p}}
We define the
full subcategory $\flc$ of $[MF^{gr}( S, W)]$ to have objects those
graded matrix factorizations $(E_1 \xra{e_1} E_0 \xra{e_0} E_1(1))$
such that $((\coker e_1)_\p)_0$ is a free $(S_\p)_0$-module for all $p \in
\Proj S$. Note that $\flc$ is a thick
subcategory.

There is a functor $\widetilde{(-)}:
[MF^{gr}( S, W)] \to [MF(\ps Q, \cO(1), W)]_\naive$ which sends a matrix
factorization $E_1 \to E_0 \to E_1(1)$ to $\widetilde{E_1} \to
\widetilde{E_0} \to \widetilde{E_1(1)}$. 

\begin{prop}
\label{graded_equiv}
The functor $\widetilde{(-)}:
[MF^{gr}( S, W)] \to [MF(\ps Q, \cO(1), W)]_\naive$ induces an
equivalence
\[ \widetilde{(-)}:
[MF^{gr}( S, W)]/\flc \to [MF(\ps Q, \cO(1), W)].
\]
\end{prop}

\begin{proof}
We first note that $\widetilde{(-)}:
[MF^{gr}( S, W)] \to [MF(\ps Q, \cO(1), W)]_\naive$ has a fully
faithful right
adjoint
\[ \Gamma_*:  [MF(\ps Q, \cO(1), W)]_\naive \to [MF^{gr}( S, W)]\]
that sends $(\cE_1 \to \cE_0 \to \cE_1(1))$ to $(\Gamma_*( \cE_1) \to
\Gamma_*( \cE_0) \to \Gamma_*( \cE_1(1))$. Here, for a coherent sheaf $\cE$,
we set $\Gamma_*(\cE ) := \bigoplus_{n \in \Z} \Gamma( \ps Q,
\cE(n) )$.

Recall from \ref{our_equiv} that \[[MF(\ps Q, \cO(1), W)] \cong
\frac{[MF(\ps Q, \cO(1), W)]_\naive}{\text{locally
    contractible objects}}.\] One checks that the adjoint pair of
functors restricts to an adjoint pair on $\flc$ and the subcategory of
${\text{locally
    contractible objects}}$, and thus induces an adjoint pair of
functors between $[MF^{gr}( S, W)]/\flc$ and  $[MF(\ps Q, \cO(1),
W)]$. It follows that the induced functor 
\[\Gamma_*: [MF(\ps Q, \cO(1),
W)] \to [MF^{gr}( S, W)]/\flc\] is also fully faithful, and the kernel
of $\widetilde{(-)}:
[MF^{gr}( S, W)]/\flc \to [MF(\ps Q, \cO(1), W)]$ is zero. The result now follows formally;
see e.g.\ the proof of 2.1 in \cite{MR2437083}.
\end{proof}

\subsection{Construction of a matrix factorization}
\label{const_mf}
As before $S$ is
  the graded $Q$-algebra $Q[T_1, \dots,
  T_c]$, and $|T_i| = 1$. Let $S(j)$ denote
  the graded free $S$-module with $S(j)_i = S_{i+j}$, and for a graded
  $S$-module $E$ we
set $E(j) = E \otimes_S
  S(j)$.

Given a finitely generated $R$-module $M$ with finite projective
dimension over $Q$, pick a finite projective resolution $G$ of $M$ as a
$Q$-module, and a system of higher homotopies $\bs{\sigma} = \{\sigma^J\}$.
We define the finitely generated graded projective $S$-modules $E_1$
and $E_0$ as follows:
$$
E_1 := \bigoplus_{j\geq 0} G_{2j+1} \otimes_Q S(j) \text{ and } E_0 :=
  \bigoplus_{j\geq 0} G_{2j} \otimes_Q S(j).
$$
For $J \in \N^c$, consider the maps 
\[ G_{2j+1} \otimes S(j) \xra{\sigma_J \otimes T^J} G_{2j +2|J|}
\otimes S(j + |J|)\]
\[ G_{2j} \otimes S(j) \xra{\sigma_J \otimes T^J} G_{2j +2|J|-1}
\otimes S(j + |J|).\]

We use these to define homogeneous maps $g_1: E_1
  \to E_0$ and $g_0: E_0 \to E_1(1)$, component-wise, as:
  \[ (g_1)_j = \sum_{J \in \N^c} \sigma^J \otimes T^J: G_{2j+1}
  \otimes_Q S(j) \to \bigoplus_i G_{2i} \otimes_Q S(i) = E_0\]
\[ (g_0)_j = \sum_{J \in \N^c} \sigma^J \otimes T^J : G_{2j}
  \otimes_Q S(j) \to \bigoplus_i G_{2i-1} \otimes_Q S(i) = E_1(1).\]

Using the defining properties of the system of higher homotopies, one checks:
\begin{lem} There are equalities:
$$
\begin{aligned}
 g_0 \circ g_1 & = \sum_{i = 1}^c f_i
 \otimes T_i = 1_{E_1} \otimes W \\ 
g_1(1) \circ g_0 & = \sum_{i = 1}^c f_i
 \otimes T_i = 1_{E_0} \otimes W. \\
\end{aligned}
$$
\end{lem}
This shows that $E = E(M, G, \bs{\sigma}) := (E_1
\xra{g_1} E_0 \xra{g_0} E_1(1))$ is an
object in $[MF^{gr}(S, W)]$.

\begin{defn} \label{def:MFHH}
For $G$ and $\bs{\sigma}$ as above, we set $\bE =
\bE(M,G,\bs{\sigma})$ to be the
object $\widetilde{E(M, G, \bs{\sigma})}$ of $[MF(\ps Q, \cO(1),
W)]$. Explicitly, $\bE = (\cE_1 \xra{e_1} \cE_0 \xra{e_0} \cE(1))$ with
\begin{align*}
  \cE_1 = \widetilde{ E_1 } = \bigoplus_{j\geq 0}G_{2j+1} \otimes_Q
  \cO_{\ps Q}(j) \quad & \quad \cE_0 = \widetilde{E_0} =
  \bigoplus_{j\geq0} G_{2j} \otimes_Q \cO_{\ps Q}(j) \\
e_1 = \widetilde{g_1}: \cE_1 \to \cE_0 \quad  & \quad
e_0 = \widetilde{g_0}: \cE_0 \to \cE_1(1).
\end{align*}
\end{defn}

\begin{prop} 
\label{higher_hom_res}
Let $M$ be a finitely generated 
$R$-module that has finite projective dimension over $Q$. Let $G$ be a finite projective 
  $Q$-resolution of $M$, $\bs{\sigma}$ a system of higher
  homotopies on $G$, and $\bE = \bE(M, G, \bs{\sigma})$ the matrix
  factorization constructed above. Then the complex $F( \RQ^\sharp \bE )$,
  where  $\RQ^\sharp \bE$ is defined
in \ref{def:flat} and $F$ is defined in \ref{def:F},
  is exactly the standard resolution $G\{\boldsymbol{\sigma} \}$
  constructed from $G$ and $\boldsymbol{\sigma}$. In particular,
there is an isomorphism in $\Dsing(R)$
$$
\Psi( \bE) \cong M 
$$
where $\Psi: [MF(\ps Q, \cO(1), W)] \to \Dsing(R)$ is the functor of Corollary \ref{orlovcor}.
\end{prop}

\begin{proof}
First note that $\RQ^*( G_k \otimes_Q \cO_{\ps Q}(j) ) \cong \overline{G_k}
\otimes_R \cO_{\ps R}(j)$, where $\RQ: \ps R \into \ps Q$ is the
canonical inclusion. Thus, we have that
\[ \RQ^\sharp( \bE ) = \cdots \to \bigoplus_j \overline{G}_{2j+4} \otimes_R
\cO(j) \to \bigoplus_j \overline{G}_{2j+3} \otimes_R
\cO(j) \to \bigoplus_j \overline{G}_{2j+2} \otimes_R
\cO(j) \to 0\] 
where the last term is in cohomological degree $c - 1$. Applying the
functor $F$ we have
\[
  F(\RQ^\sharp( \bE )) = \cdots \to \rH^{c-1}(\ps R, \bigoplus_j \overline{G}_{2j+4} \otimes_R
\cO(j)) \to
\rH^{c-1}(\ps R, \bigoplus_j \overline{G}_{2j+3} \otimes_R
\cO(j))\] 
\[\to \rH^{c-1}(\ps R, \bigoplus_j \overline{G}_{2j+2} \otimes_R
\cO(j) ) \to 0, \]
where the last term is located in cohomological degree $2c - 2$. In
general, we have that $ F(\RQ^\sharp( \bE ))^{2i} = \bigoplus_j G_{2j
  +2(c-i)} \otimes \cO(j)$. Now, by \cite[Thm. III.5.1]{MR0463157},
for all $j$, there is
a natural isomorphism
\begin{align*} \label{E719}
\rH^{c-1}( {\ps R}, \cO_{{\ps R}}(j)) & \cong 
\Hom R {R[T_1, \ldots, T_c]^{-j -c}} R \\ & \cong \Hom R {R[\chi_1, \ldots, \chi_c]^{-2j - 2c}} R = D^{2j + 2c}.
\end{align*}
Since each $G_i$ is a projective $R$-module, this gives that
\begin{align*}
  F(\RQ^\sharp( \bE )) \cong \cdots & \to \bigoplus_j \overline{G}_{2j+4} \otimes_R
D^{2j+2c}  \to 
\bigoplus_j \overline{G}_{2j+3} \otimes_R
D^{2j+2c} \\ & \to \bigoplus_j \overline{G}_{2j+2} \otimes_R
D^{2j+2c} \to 0
\end{align*}
where the last term is in cohomological degree $2c - 2$. However, $D^j
= 0$ unless $j \leq 0$ and $G_i = 0$ unless $i \geq 0$, and thus the
first non-zero terms of this complex are
\[ 
\cdots \to (G_2 \otimes D^0) \oplus (G_0 \otimes D^2) \to G_1 \otimes D^0 \to G_0
\otimes D^0 \to 0 \]
where $G_0 \otimes D^0$ is in degree 0.
As a graded $R$-module this is exactly $\overline{G} \otimes D$, and
one may readily verify that the differentials are the
same. 

Thus $F(\RQ^\sharp(\bE) )$ is a free resolution of $M$ and so is
isomorphic to $M$ in $\Db(R)$. Let $m$ be an integer which is
sufficiently small for $\bE$. Then by \ref{prop:F} we have that
$F(\RQ^\sharp(\bE))^{\leq m - c +1} = F(\RQ^\sharp(\bE)^{\leq m}) \cong
\Psi( \bE ) \in \Dsing(R).$ Since $F (\RQ^\sharp(\bE))$ is projective in
each degree, the canonical map
\[ F(\RQ^\sharp(\bE)) \to F(\RQ^\sharp(\bE))^{\leq m - c +1}\]
is an isomorphism in $\Dsing(R)$. Thus $M \cong \Psi(\bE) \in \Dsing(R).$
\end{proof}

\begin{cor}
There is an equivalence
\[\varphi: [MF^{gr}( S, W)]/\flc \to \rDsg( Q \onto R), \]
where $\flc$ is the full subcategory of $[MF^{gr}( S, W)]$ with
objects those $(E_1 \xra{e_1} E_0 \xra{e_0} E_1(1))$
such that $((\coker e_1)_\p)_0$ is a free $(S_\p)_0$-module for all $p \in
\Proj S$. If $M$ is an $R$-module with a finite projective resolution $G$ over $Q$
and $\bs{\sigma}$ is a system of higher homotopies on $G$, then
$\varphi( E(M , G, \bs{\sigma})) \cong M \in \rDsg( Q \onto R).$
\end{cor}

\begin{proof}
  Define $\varphi$ to be the composition of the equivalence $\widetilde{(-)}$ of \ref{graded_equiv}
  with the equivalence $\Psi$ of \ref{orlovcor}. Let
  $E = E(M , G, \bs{\sigma})$ be the object of $[MF^{gr}(S,W)]$
  defined above. By definition, $\widetilde{E} = \bE( M, G,
  \bs{\sigma})$, and \[\varphi(E)
  := \Psi( \widetilde{E} ) = \Psi( \bE( M, G,
  \bs{\sigma})) \cong M\]
where the last isomorphism is by Proposition \ref{higher_hom_res}.
\end{proof}

\section{Properties of matrix factorizations}
\label{props_mf}
In this section we state and prove basic isomorphisms using matrix
factorizations and discuss the support of a matrix
factorization. These properties will translate directly to the
properties of stable support sets described in Theorem \ref{introthm4} of the
Introduction. 

In this section, as there is nothing gained in working over $\ps Q$, we work in the
generality that $X$ is a
Noetherian separated
scheme and $\cL$ is a line bundle on $X$. As a matter of convenience,
we write $\cF(1)$ for  $\cF
\otimes_{\cO_X} \cL$ even though $\cL$ is not assumed to be very ample.

\begin{defn}
\label{def_ten_prod_mf} Suppose $W$ and $V$ are global sections of $\cL$ and
let
$\E =  (\cE_1 \xra{e_1} \cE_0 \xra{e_0} \cE_1(1))$ and  $\F =
(\cF_1 \xra{f_1} \cF_0 \xra{f_0} \cF_1(1))$ are objects of
$MF(X, \cL, W)$ and $MF(X, \cL, V)$, respectively.

\begin{enumerate}
\item Their {\em tensor product} is the object $\E \otimes_{MF} \F$ of $MF(X, \cL, W + V)$
  given by
  \[ \left (
    \begin{array}{c}
      \cE_0 \otimes \cF_1 \\
      \oplus\\
      \cE_{1} \otimes \cF_{0}
    \end{array}
    \xra{\partial^{\bE \otimes \bF}_1}
    \begin{array}{c}
      \cE_0 \otimes \cF_0 \\
      \oplus\\
      (\cE_{1} \otimes \cF_{1})(1)
    \end{array}
    \xra{\partial^{\bE \otimes \bF}_0} \left(\begin{array}{c}
        \cE_0 \otimes \cF_1 \\
        \oplus\\
        \cE_{1} \otimes \cF_{0}
      \end{array}\right )(1) \right).
  \]
  The differentials are given by the formulas
$$
\partial^{\bE \otimes \bF}_1 =
\begin{bmatrix}
  1 \otimes f_1 & e_1 \otimes 1 \\
  e_0 \otimes 1  & -1 \otimes f_0 \\
\end{bmatrix}
\quad \text{ and } \quad
\partial^{\bE \otimes \bF}_0 =
\begin{bmatrix}
  1 \otimes f_0 & (e_1 \otimes 1)(1) \\
  e_0 \otimes 1  & (-1 \otimes f_1)(1) \\
\end{bmatrix}
$$
using the canonical isomorphisms
$$
\left(\begin{array}{c}
    \cE_0 \otimes \cF_1 \\
    \oplus\\
    \cE_{1} \otimes \cF_{0}
  \end{array}\right )(1) 
\cong
\begin{array}{c}
  (\cE_0 \otimes \cF_1)(1) \\
  \oplus\\
  (\cE_{1} \otimes \cF_{0})(1)
\end{array} 
$$
and $\cE_i(1) \otimes \cF_j \cong (\cE_i \otimes \cF_j)(1) \cong \cE_i
\otimes \cF_j(1)$.
\\
\item Their \emph{$\operatorname{Hom}$-object} is the object 
$\HomMF(\bE, \bF)$
of
  $MF(X, \cL, V-W)$ given by
\begin{equation*}
{\small
\begin{matrix}
  \uHom(\cE_0, \cF_1) \\
  \oplus \\
  \uHom(\cE_1, \cF_0(-1))
\end{matrix}\xra{\partial^{-1}}
\begin{matrix}
  \uHom(\cE_0, \cF_0) \\
  \oplus \\
  \uHom(\cE_1, \cF_1)
\end{matrix}
\xra{\partial^{0}}
\left(\begin{matrix}
  \uHom(\cE_0, \cF_1) \\
  \oplus \\
  \uHom(\cE_1, \cF_0(-1))
\end{matrix}
\right)(1)
}
\end{equation*}
with the differentials defined using the same formulas as in Definition \ref{def_hom_complex}, which
represents the special case $W=V$.
\end{enumerate}

\end{defn}

\begin{rem} We will eventually assume $X$ is regular. If
  $X$ is not regular, these definitions should be view as
  ``non-derived''. 
\end{rem}

\begin{rem} If $W + V = 0$, we may interpret $\bE \otimes_{MF}
  \bF$ as an object of $TPC(X, \cL)$, i.e.\ a twisted periodic
  complex. See
  \ref{twisted_periodic_def} for the definition.
If $W=V=0$, we obtain a tensor operator
  for the category $TPC(X, \cL)$.
\end{rem}

\begin{prop} \label{assoc_comm} Let $W,V,$ and $U$ be global sections of $\cL$,
  and $\bE$, $\bF$, and $\bG$ be objects of $MF(X, \cL, W)$,
  $MF(X, \cL, V),$ and $MF(X, \cL, U)$, respectively. There are isomorphisms:
  \begin{enumerate}
  \item  $\bE \otimes_{MF} \bF \cong \bF \otimes_{MF} \bE$
in $MF(X, \cL, W+V)$, 

\item $(\bE \otimes_{MF} \bF) \otimes_{MF} \bG \cong \bE \otimes_{MF} (\bF
\otimes_{MF} \bG)$
in $MF(X, \cL, W+V+U)$, and  

\item $\HomMF( \bE \otimes_{MF} \bF, \bG ) \cong \HomMF( \bE, \HomMF(
  \bF, \bG))$ in $MF(X, \cL, U - V - W )$.

\end{enumerate}

\end{prop}

\begin{proof} The isomorphisms of underlying locally free sheaves are
  the natural ones in each degree, taking care to observe the usual
  sign convention. For example,  the map
  $\bE \otimes_{MF} \bF \to \bF \otimes_{MF} \bE$ sends a section $a \otimes
  b$ of $\cE_i \otimes \cF_j$ to $(-1)^{ij}b \otimes a$.
A straightforward, but tedious, check shows that these isomorphisms
commute with the differentials of the matrix factorizations.
\end{proof}

\begin{defn} 
The {\em dual} of an object $\E =  (\cE_1 \xra{e_1} \cE_0 \xra{e_0}
 \cE_1(1))$ in $MF(X, \cL, W)$ is the object of $MF(X, \cL, -W)$ given by
\[ \E^\vee = \cE_1(1)^\vee \xra{ -e_0^\vee} \cE_0^\vee \xra{e_1^\vee}
\cE_1(1)^\vee(1),
\]
where $(-)^\vee$ denotes the functor $\uHom_{\cO_X}(-, \cO_X)$ and we use the
canonical isomorphisms
$\cE_1(1)^\vee(1) \cong \cE_1^\vee(-1)(1) \cong \cE_1^\vee$.
\end{defn}

\begin{rem} \label{double_dual}Equivalently, $\bE^\vee$ is $\HomMF(\bE, \cO_X)$ where
  $\cO_X$ denotes the matrix factorization $(0 \to \cO_X \to 0)$
  belonging to  $MF(X, \cL, 0)$. We also note that there is a natural isomorphism $(\bE^\vee)^\vee \cong \bE$. 
\end{rem}

\begin{prop} \label{P72a}
Let $\E$ and $\F$ be objects of $MF(X, \cL, W)$ and $MF(X,\cL,V)$,
respectively. There is a natural isomorphism
in $MF(X, \cL, V-W)$
$$
\E^\vee \otimes_{MF} \F \map{\cong} \HomMF(\bE, \bF).
$$
\end{prop}

\begin{proof}
Recall that for locally free coherent sheaves $\cE$ and $\cF$, there is a canonical isomorphism $\cE^\vee \otimes \cF \map{\cong}
\uHom(\cE, \cF)$ given on sections by  $\delta \otimes f \mapsto (e \mapsto
\delta(e) \cdot f)$.
In homological degree $1$, the isomorphism we seek is the direct sum of these  canonical ones:
$$
\begin{bmatrix}
can & 0 \\
0 & can \\
\end{bmatrix}:
\begin{matrix}
\cE_0^\vee \otimes \cF_1 \\
\oplus \\
\cE_1^\vee \otimes  \cF_0(-1) \\
\end{matrix}
\map{\cong}
\begin{matrix}
\uHom(\cE_0, \cF_1)  \\
\oplus \\
\uHom(\cE_1, \cF_0(-1)).
\end{matrix}
$$
In degree $0$, it is
$$
\begin{bmatrix}
can & 0 \\
0 & -can \\
\end{bmatrix}:
\begin{matrix}
\cE_0^\vee \otimes \cF_0 \\
\oplus \\
\cE_1(1)^\vee \otimes  \cF_1(1) \\
\end{matrix}
\map{\cong}
\begin{matrix}
\uHom(\cE_0, \cF_0)  \\
\oplus \\
\uHom(\cE_1, \cF_1),
\end{matrix}
$$
where we have also used the canonical isomorphism 
$\uHom(\cE_1(1), \cF_1(1)) \cong \uHom(\cE_1, \cF_1)$.
We omit the straightforward verification that the differentials commute
with these isomorphisms.
\end{proof}

\begin{rem} In the case $W=V$, the proposition gives an isomorphism of
  twisted periodic complexes.
\end{rem}

\begin{prop} \label{dualten}
For matrix factorizations
  $\bE$ and $\bF$ in $MF(X, \cL, W)$ and $MF(X, \cL, V)$,
  respectively, there are natural isomorphisms
  \begin{enumerate}
  \item $
(\bE \otimes_{MF} \bF)^\vee \cong \bE^\vee \otimes_{MF} \bF^\vee
$
in $MF(X, \cL, -W-V)$ and
\item $\HomMF(\bE, \bF)^\vee \cong \HomMF(\bF, \bE)$ in  $MF(X, \cL, W-V)$.
\end{enumerate}
\end{prop}

\begin{proof} 
Part (1) follows from $\operatorname{Hom}$-tensor adjointness and Proposition \ref{P72a}.

For (2), we have
the following isomorphisms:
\begin{align*}
\HomMF(\bE, \bF)^\vee 
&\cong
( \bE^\vee \otimes_{MF} \bF)^\vee
\cong
(\bE^\vee)^\vee \otimes_{MF} \bF^\vee \\
\cong \bE \otimes_{MF} \bF^\vee
&\cong
\bF^\vee \otimes_{MF} \bE \cong
\HomMF(\bF, \bE)
 \end{align*}
by \ref{P72a}, part (1) of this Proposition, \ref{double_dual}, \ref{assoc_comm}(2), and \ref{P72a}, respectively.
\end{proof}

\begin{prop}
\label{switch_isom} Let $W_1, \dots, W_4$ be global sections of $\cL$
and let $\bE_i$ be an object of $MF(X, \cL, W_i)$, for $i
  =1, \dots 4$. There is an
  isomorphism 
\[ 
\HomMF(\bE_1,\bE_2) \otimes_{MF} \HomMF(\bE_3, \bE_4) \cong 
\HomMF(\bE_1, \bE_4) \otimes_{MF} \HomMF(\bE_2, \bE_3) 
\]
in $MF(X, \cL, W_2-W_1+W_4-W_3)$.
\end{prop}

\begin{proof}
This follows immediately from Propositions \ref{assoc_comm} and
\ref{P72a}.
\end{proof}

\begin{defn} 
 The \emph{support} of a complex $\cP$ of quasi-coherent sheaves on a
 scheme $X$ is
$$ 
\supp \cP = \{x \in X \, | \, \text{ $\cP_x$ is not exact} \} = 
\bigcup_{i \in \Z} \supp \cH^i( \cP ). 
$$ 
For an
  object $\cP$ of $TPC(X, \cL) = MF(X, \cL, 0)$, i.e., a twisted
  periodic complex, we have that 
\[ \supp \cP = \supp \cH^0(
  \cP ) \cup \supp \cH^1( \cP ) \] since
  $\cH^{i + 2} \cong \cH^i (1 )$. In particular the support of a
  twisted periodic complex is
  a closed subset of $X$.
\end{defn}

Recall that if $\cL = \cO_X$ is the
trivial line bundle, then $TPC(X, \cO_X)$ is the category of
$\Z/2$-graded complexes of locally free coherent sheaves on $X$. 

\begin{lem} \label{lem:istar} Assume $X$ is a regular Noetherian separated
  scheme and $\cL$ is a line bundle on $X$.
For a point $x \in X$, let $k(x)$ denote its residue field and let
 $i_x: \Spec k(x)  \to X$ be the canonical map. 
Given $\cP \in TPC(X, \cL)$, a point $x \in X$ belongs to
$\supp(\cP)$ if and only if the $\Z/2$-graded complex of $k(x)$-vector
spaces $i_x^*(\cP)$ is not exact.
\end{lem}

\begin{proof}
We can reduce immediately to the assertion that for a regular local
ring $(S, \fm, k)$ and a (possibly unbounded) complex
$\cP$ of finitely generated free $S$-modules, $\cP$ is exact 
if and only if $k \otimes_S \cP$ is exact.

Let $x_1, \ldots, x_n \in S$ be a regular system of parameters.
The long exact sequence in cohomology associated to the short exact sequence 
$$
 0 \to \cP \map{x_n} \cP \to \cP/x_n\cP \to 0
$$
of complexes
and Nakayama's Lemma give that $\cP$ is exact if and only if
$\cP/x_n\cP$ is exact. The result follows by induction on $n$.
\end{proof}

\begin{rem} The assumption that $X$ is regular is essential for Lemma
  \ref{lem:istar}.
\end{rem}

\newcommand{\cQ}{\mathcal{Q}}
\begin{prop} \label{prop:cap}
Assume $X$ is a regular Noetherian separated
  scheme and $\cL$ is a line bundle on $X$.
For objects
 $\cP$ and $\cQ$ of $TPC(X, \cL)$, we have
\[ \supp ( \cP \otimes_{MF} \cQ ) = \supp \cP \cap \supp \cQ. \]
\end{prop}

\begin{proof}
For any  $x \in X$ we have
\[
i_x^*( \cP \otimes_{MF)} \cQ)  \cong
i_x^*( \cP ) \otimes_{k(x)}^{\Z/2} i_x^*( \cQ),
\]
where $\otimes_{k(x)}^{\Z/2}$ denotes the tensor product for $\Z/2$-graded complexes of
$k(x)$-vector spaces (i.e., for the category $TPC(\Spec k(x), \cO)$).
For any $\Z/2$-graded complex of $k(x)$-vector spaces $V$,
we have 
$V \cong \cH^0(V) \oplus \cH^1(V)[1]$.
It follows that for a pair $V, W$ of such complexes, we have
$$
\begin{aligned}
V \otimes_{TPC} W
\cong  \,
& \left ( \cH^0(V) \otimes_k \cH^0(W) \right ) \, \oplus \,  \left ( \cH^1(V) \otimes_k
\cH^1(W) \right ) \, \oplus \\ 
& \left ( \cH^0(V) \otimes_k \cH^1(W)[1] \right ) \, \oplus \, \left (
  \cH^1(V) \otimes_k \cH^0(W)[1] \right ).
\end{aligned}
$$
In particular, $V \otimes_{TPC} W$
is exact if and only if $V$ or $W$ is  exact.
The result follows from Lemma \ref{lem:istar}.
\end{proof}

\begin{prop}
\label{supp_dual}
 Assume $X$ is a regular Noetherian separated
  scheme and $\cL$ is a line bundle on $X$. For an object $\cP$ of $TPC(X, \cL)$, there is an equality
\[ \supp(\cP) = \supp(\cP^*). \]
\end{prop}

\begin{proof} 
We have the isomorphism
$$
i^*_x(\cP^*) \cong \left(i_x^*(\cP)\right)^*
$$
of $\Z/2$-graded complexes of $k(x)$-vector spaces. 
For any such complex $V$, we have
$V \cong \cH^0(V) \oplus \cH^1(V)[1]$ and 
$V^* \cong \cH^0(V)^* \oplus \cH^1(V)^*[1]$. Thus
$V$ is exact if and only if $V^*$ is exact, using Lemma \ref{lem:istar}.
\end{proof}

\section{Stable support} \label{suppsec}
In this section, returning to the context and assumptions of
\S \ref{assumptions}, we study the properties of stable support defined in
\ref{suppdef}.  In
particular we prove Theorems \ref{introthm4} and \ref{introthm5} from
the Introduction. 

Recall that the stable support set of a pair $(M,N)$ of 
complexes of $R$-modules with bounded and finitely
generated cohomology and $M$ perfect over $Q$, is defined to
be
$$
V_Q^{\vf}(M,N) =
\supp \widetilde{ \Ext {ev} R M N }
\cup
\supp \widetilde{ \Ext {odd} R M N } \subseteq \ps R.
$$
If $M$ and $N$ are both perfect over $Q$, then
it follows from \ref{stable_ext_hommf} that the stable support of $(M,N)$
may be computed as the support of a twisted-periodic complex:
$$
V_Q^{\vf}(M,N) = \supp \HomMF(\bE_M, \bE_N).
$$

In the following $Q$ is assumed to be regular so that we may use
Propositions \ref{prop:cap} and \ref{supp_dual}.
\begin{thm} \label{suppthm} 
Let $Q$ be a regular Noetherian  ring of finite Krull dimension and
let $R = Q/(f_1, \cdots, f_c)$ for a $Q$-regular sequence $f_1, \dots,
f_c$.
For complexes of $R$-modules
$M$, $N$, $M'$, and $N'$ with bounded finitely generated cohomology, we have
\begin{enumerate}
\item $V_Q^{\vf}(M,N) = \emptyset$ if
and only if $\Ext n R M N = 0$ for all $n \gg 0$,
\item $V_Q^{\vf}(M,N) \cap V_Q^{\vf}(M',N') = V_Q^{\vf}(M,N') \cap V_Q^{\vf}(M',N)$, and 
\item $V_Q^{\vf}(M,N) = V_Q^{\vf}(M,M) \cap V_Q^{\vf}(N,N) = V_Q^{\vf}(N,M)$.
\end{enumerate}
\end{thm}

\begin{proof}
The first assertions follows from the definition, using that for a
finitely generated graded
$R[T_1, \dots, T_c]$-module $E$, the associated coherent sheaf
$\widetilde{E}$ vanishes if and only if
$E_n = 0$ for all $n \gg 0$ if and only if $\supp \widetilde{E} = \emptyset$.

Since $Q$ is regular every $Q$-module has finite projective dimension.
In particular $V_Q^{\vf}(M,N) = \supp \HomMF(\bE_M, \bE_N)$ and we may
use the results of the previous section. The second assertion follows
from \ref{switch_isom} and
\ref{prop:cap}. Using the second assertion we have that $V_Q^\vf(M,M)
\cap V_Q^{\vf}(N,N) = V_Q^\vf(M,N) \cap V_Q^\vf(N,M)$. By
\ref{dualten}, \ref{prop:cap}, and
\ref{supp_dual} there is an equality $V_Q^\vf(M,N) = V_Q^\vf(N,M)$.
\end{proof}

Although Theorem \ref{suppthm} assumes $Q$ is regular, we do not make this assumption in the rest of this section.

The \emph{singular locus} of a Noetherian scheme $Z$ is
$$
\sing(Z) = \{ z \in Z \,
  | \, \text{ $\cO_{Z,x}$ is not a regular local ring} \}.
$$
Recall
that $\YQ: Y \into \ps Q$ denotes the zero subscheme of $W
= \sum_i f_iT_i \in \cO_{\ps Q}(1)$. By \cite[10.2]{1105.4698} there
is a containment $\sing (Y) \subseteq \ps R$.
\begin{lem}
\label{supp_is_in_sing}
For every pair $(M,N)$ of 
complexes of $R$-modules with bounded and finitely
generated cohomology and $M$ perfect over $Q$, there is a containment
\[ V_Q^{\vf}(M,N) \subseteq \sing(Y). \]
\end{lem}
\begin{proof}
By \ref{suppcor} and \ref{suppdef}, we have that $V_Q^{\vf}(M,N) =
\supp \sExt^0_{\cO_Y}( \cM, \cN)$, where $\cM = \RY_* \pi^* M$ and $\cN = \RY_* \pi^* N$ in the
notation of \S \ref{assumptions}. For any point $y \in Y$, we have 
\[\sExt^0_{\cO_Y}( \cM, \cN)_y \cong \sExt^0_{\cO_{Y,y}}(
\cM_y, \cN_y) = \rHom_{\Dsg(\cO_{Y,y})}( \cM_y, \cN_y)\]
where the first isomorphism is by \ref{prop:1450} and the second is by
Definition \ref{stable_ext_defn}. If $y \notin
\sing(Y)$, then $\Dsg(\cO_{Y,y}) = 0$ and hence $\sExt^0_{\cO_Y}( \cM,
\cN)_y  = 0$.
\end{proof}

\subsection{Relation to other notions of support}
\newcommand{\AB}{\operatorname{AB}}
\newcommand{\BIK}{\operatorname{BIK}}
\newcommand{\St}{\operatorname{St}}

In \cite{AvBu00} Avramov and Buchweitz defined a notion of support under the additional assumption that $Q$ is local. 
In that case, where for simplicity we assume the residue field $k$ of $Q$ is algebraically closed,
the
  \emph{AB-support variety} of $M$ and $N$, which we write $V_Q^{\vf}(M,N)^{\AB}$,
  is the union of the supports in $\A_{k}^c$ of the graded $k[T_1,
  \dots, T_c]$-modules $\Ext
  {ev} R M N \otimes_R k$ and $\Ext {odd} R M N \otimes_R
  k$.
It follows immediately from the definitions that 
 $V_Q^{\vf}(M,N)^{\AB}$ is the cone of the closed subset 
\[
V_Q^{\vf}(M,N)
  \times_{\Spec R} \Spec k \subseteq
\ps k.
\]
In particular, we deduce from Theorem \ref{suppthm} that 
  the analogous formulas for $V_Q^{\vf}(M,N)^{\AB}$ hold. These were
  first established
  in \cite{AvBu00}.

Benson, Iyengar, Krause defined a support for all objects in the infinite
completion $\mathsf{K}( \operatorname{Inj} R)$ of $\Db(R)$ in \cite{BIK08}. Here
$\mathsf{K}( \operatorname{Inj} R)$ is the homotopy category of
injectives; see \cite{Kr05} for further details. For $M$ an object in
$\Db(R)$, the support in the sense of \cite{BIK08}, which we
write as $V_Q^\vf(M)^{\BIK}$, may be computed as \[V_Q^\vf(M)^{\BIK} = \supp_{R[T_1, \ldots, T_c]} \Ext
{ev}  R M M \, \cup \, \supp_{R[T_1, \ldots, T_c]} \Ext
{odd}  R M M \subseteq \mathbb{A}^c_R.\]
Thus $V_Q^\vf(M,M)$ is equal to the image of $V_Q^\vf(M)^{\BIK}$ under
the canonical map $\mathbb{A}^c_R \setminus \{ 0 \} \to \ps R$.

Finally, in \cite{1105.4698}, for a
Noetherian separated scheme $Y$ with hypersurface singularities,
Stevenson defined a
support for all objects in the infinite completion
$\mathsf{K}_{\operatorname{ac}}( \operatorname{Inj} Y)$ of $\Dsg(Y)$;
here $\mathsf{K}_{\operatorname{ac}}( \operatorname{Inj} Y)$ is the
homotopy category of acyclic complexes of injective quasi-coherent
sheaves on $Y$. In
case $Y \into \ps Q$ is the zero subscheme of $W \in \cO(1)$,
where we assume that $Q$ is regular, then the support of a  coherent
sheaf $\cM$ on $Y$, which we write as $V_Q^\vf(\cM)^{\St}$, is
\[ V_Q^\vf(\cM)^{\St} = \{ \, x \in \Sing Y \, | \, \cM_x \neq 0 
\text{ and } \pd_{\cO_{Y,x}}
  \cM_x = \infty \} \subseteq \Sing Y. \]
Indeed, by \cite[8.9]{1105.4692}, support may be
computed affine locally, and then the proof of
\cite[5.12]{1105.4698} shows that the support of compact objects on affine hypersurfaces
takes the form above. By \cite[10.2]{1105.4698} there is a containment
$\Sing Y \subseteq \ps R$.

We claim that for $M$ an object of $\Db(R)$,
there is an equality
\[ V_Q^\vf(M, M) =  V_Q^\vf(\cM)^{\St}, \]
where $\cM = \RY_* \pi^*M$. Indeed, by \ref{suppcor} and the
proceeding remark we have that
\[V_Q^\vf(M, M) = \supp
\suExt^{0}_{\cO_Y}(\cM, \cM).\]
The equality of support sets now follows from the isomorphism $\suExt^{0}_{\cO_Y}(\cM,
\cM)_x \cong \sExt^0_{\cO_{Y,x}}( \cM_x, \cM_x)$ and
\cite[4.2]{AvBu00} which implies that $\cM_x$ has finite projective
dimension if and only if $\sExt^0_{\cO_{Y,x}}( \cM_x, \cM_x) = 0$.

\subsection{Realization}
The following theorem answers a question posed to us by Avramov.

\begin{thm} \label{thm_realization} 
Let $Q$ be a Noetherian  ring of finite Krull dimension and
let $R = Q/(f_1, \cdots, f_c)$ for a $Q$-regular sequence $f_1, \dots,
f_c$. For every closed subset $C$ of $Y$ that is contained in
  $\sing(Y)$ and satisfies $C \cap \sing( \ps Q) = \emptyset$, there is an $R$-module $M$ such that $M$ has finite
  projective dimension over $Q$ and satisfies $V_Q^{\vf}(M,M) = C$. In
  particular if $Q$ is regular, then the result holds for every 
  $C \subset \sing(Y)$ and, in this situation, $M$ may be chosen to be
  an MCM $R$-module. 
\end{thm} 

\begin{rem}
This recovers a theorem of Bergh \cite{Be07} and Avramov-Iyengar
 \cite{AvIy07}. When $Q$ is regular, Theorem \ref{thm_realization}
 follows from \cite[7.11]{1105.4698} using the relation, as sketched above, of support defined in
 \emph{loc. cit.}\ to stable support.
\end{rem}

\begin{proof} Let $C$ be a
  closed subset of
  $\sing(Y)$ such that $C \cap \sing( \ps Q) = \emptyset$. Regard $\cO_C$ as a coherent sheaf on $Y$ where $C$ is
  given the reduced subscheme structure. The coherent sheaf
$\YQ_* \cO_C$
  is perfect on $\ps Q$, since for any point $x$ in $\sing( \ps Q)$, we have that $(\YQ_*
  \cO_C)_x = 0$, and for any point $x$ not in $\sing( \ps Q)$, every
finitely generated $\cO_{\ps Q, x}$-module has finite projective
dimension. 

Thus $\cO_C$ belongs to $\rDsg( Y \into \ps Q)$ and so
determines an object
$$
M \in \rDsg(Q \onto R)
$$
via the equivalences $\Phi: \rDsg(Q \onto R) \cong \rDsg(Y \into \ps
Q) $ of Theorem \ref{orlov}. Moreover, every object of $\rDsg(Q \onto
R)$ is isomorphic to an $R$-module by \ref{complexes_isom_to_mod_rel_sing}, and so
we may assume $M$ is an $R$-module. When $Q$ is regular, every object
of $\rDsg(Q \onto R) = \Dsg(R)$ is isomorphic to an MCM $R$-module and
so we may assume $M$ is an MCM $R$-module in this case. 

Let $\cM = \RY_* \pi^* M$, in the
notation of \S \ref{assumptions}. We have an isomorphism $\cO_C \cong
\cM$ in $\rDsg(Y \into \ps Q)$ since $\RY_* \pi^*$ induces an inverse
equivalence of $\Phi$. Thus we have that $\bE_{\cM} \cong
\bE_{\cO_C}$, and hence that 
\begin{equation}\HomMF( \bE_{\cM}, \bE_{\cM}) \cong
\HomMF( \bE_{\cO_C}, \bE_{\cO_C}). \label{mf_eq1}
\end{equation} We wish to show that $V_Q^\vf(M,M)
= C$. By \ref{suppcor}, \ref{stable_ext_hommf}, and \eqref{mf_eq1}, it
is enough to show that 
\[\supp \sExt^0_{\cO_Y}(\cO_C, \cO_C) = C.\] The containment
$\supp \sExt^0_{\cO_Y}( \cO_C, \cO_C) \subseteq C$ is clear. To prove
the opposite containment, let $y \in C$ be the generic point of an
irreducible component of the reduced subscheme associated to $C$,  and observe that, since $\supp \sExt^0_{\cO_Y}( \cO_C, \cO_C)$ is a
closed subset of $\ps Q$, it suffices to prove 
$y \in \supp \sExt^0_{\cO_Y}( \cO_C, \cO_C)$.
Since $C$ is reduced and $y$ is minimal, we have $\cO_{C,y} = A/\fm$ where $(A,
\fm)$ is the local ring of $Y$ at $y$. Thus
$$
\sExt^0_{\cO_Y}( \cO_C, \cO_C)_y
\cong
\sExt^0_A( A/\fm, A/\fm).
$$
We also have that $A = B/(f)$, for some local ring $B$ with $f$ a
non-zero divisor, and $\pd_B M \li$.
Since $y \in C \subset
\sing(Y)$, $A$ is not regular and hence 
$\sExt^0_A( A/\fm, A/\fm) \ne 0$.
\end{proof}

\appendix
\section{Orlov's Theorem}\label{appendix} 
In this appendix we establish a generalization of a theorem of Orlov \cite[Theorem
2.1]{MR2437083}. This generalization does not require the schemes in
question to be defined over a ground field and drops a
smoothness assumption.
Our proof follows \emph{loc.\ cit}.\ closely;
most parts of that proof work in the more general setup. 

Let $S$ be a Noetherian separated scheme of finite Krull dimension
that has enough locally free sheaves (i.e., every coherent sheaf on
$S$ is the quotient of a locally free coherent sheaf). Let $\cE$ be a vector
bundle of rank $r$ on $S$ and let $U \in \Gamma(S, \cE)$ be a regular
section. Let $j: S' \into S$ be the zero subscheme determined by $U$.

Let $q: X = \bP( \cE ) \to S$ be the projective bundle corresponding to
$\cE$ and write $\cO_X(1)$ for the associated line bundle on
$X$.
Let $W \in \Gamma( X,
\cO_X(1) )$ be the section induced by $U$ and let $u: Y \into X$ be
the zero subscheme of $W$. Also consider the locally free sheaf $j^*
\cE$ on $S'$, which is isomorphic to the normal bundle $\cN_{S'/S}$ of
the inclusion of $S'$ into $S$. Let $p: Z = \bP( j^* \cE) \to S'$ be the corresponding
projective bundle and $\cO_Z(1)$ the canonical line bundle.

The canonical map $Z \to X$ factors through $Y \to
X$ via a map $i: Z \to Y$. There is a short
exact sequence of coherent sheaves on $Z$
\begin{equation}\label{normal_bundle_ses} 0 \to \cN_{Z/Y} \to p^* j^*
  \cE \to \cO_Z(1) \to 0,
\end{equation}
where $\cN_{Z/Y}$ is the normal bundle of
$Z$ in $Y$ and $p^* j^* \cE \to \cO_Z(1)$ is the canonical map.

These constructions are summarized in the following diagram:
\[ \xymatrix{ Z = \bP( j^* \cE) \ar[r]^(.65)i \ar[d]^p & Y \ar[r]^(.38)u & X =
  \bP( \cE ) \ar[d]_q \\
S' \ar[rr]^j & & S} \]

Since $p$ is flat, $\mathbf{L} p^* = p^*$ preserves boundedness of
complexes. Since $i$ is a closed immersion, $\R
i_* =  i_*$ preserves coherence and
boundedness. It follows that $i_* p^*$ induces a functor
\[
\Phi_Z = i_* p^*: \Db( S' ) \to \Db( Y ).
\]
Here, as in the body of the paper we write $\Db(Y)$ for the bounded
derived category of quasi-coherent sheaves on $Y$ with coherent cohomology.

We recall the definitions of relatively perfect complexes and the
relative singularity category from \ref{defn_rperf}.
\begin{defn}
Let $i: Y \into X$ be a closed immersion of subschemes of finite flat
dimension.
An object $\cF$ in $\Db(Y)$ is a
\emph{relatively perfect complex} for $i$ if
$i_* \cF$ is a perfect complex on $X$. We write $\rperf Y X$ for the full
subcategory of $\Db( Y)$ whose objects are the relatively perfect
complexes for $i$.
\end{defn}

The category $\rperf Y X$ is a thick
  subcategory of $\Db(Y)$ and there is an inclusion $\Perf Y \subseteq \rperf Y
  X$. Moreover, the induced functor
\[ \rperf Y X / \Perf Y \to \Db(Y)/\Perf Y = \Dsing(Y)\]
is fully faithful. We thus identify $\rperf Y X/\Perf Y$ with the
corresponding full subcategory of $\Dsing(Y)$.
\begin{defn}
  The \emph{relative singularity category} of $Y$ in $X$ is
\[ \rDsg( Y \subseteq X ) := \rperf Y X/\Perf Y. \]
\end{defn}

\begin{thm}
\label{Orlov_thm}
Using the notation above, the functor $\Phi_Z$ induces a functor
\[ \bar{\Phi}_Z: \rDsg(S' \subseteq S) \to \rDsg(Y \subseteq X)\]
that is an equivalence of triangulated categories. 
\end{thm}

\begin{rem}
 In \cite[Theorem 2.1]{MR2437083} $S$ is assumed to be a regular
 scheme over a field. In this case $X$ is also regular and hence there are equivalences
 $\rDsg( S' \subseteq S) \cong \Dsg(S')$ and $\rDsg( Y \subseteq X)
 \cong \Dsg(Y)$. 
\end{rem}

\begin{proof}
The functor $\Phi_Z$ has a right adjoint, written $\Phi_Z^*$,
which
is given by $\R p_* i^{\flat}$. Here $i^{\flat}$ is the right adjoint to
$\R i_*$ and is given by $\mathbb{L} i_*( - ) \otimes
\omega_{Z/Y}$, where $\omega_{Z/Y} = \wedge^{r-1} \cN_{Z/Y}[-r+1]$;
see \cite[III.7.3]{MR0222093}.
The proof of Proposition 2.2 in \cite{MR2437083} goes through verbatim to
show that the natural map $\text{id} \to \Phi_Z^* \Phi_Z$ is an
isomorphism and thus that $\Phi_Z$ is fully faithful.

Now observe that if $\cF $ is an object of  $\rperf {S'} S$ then
$\Phi_Z(\cF)$ is an object of $\rperf Y X$. Indeed, the diagram
\[ \xymatrix{ Z \ar[r]^k \ar[d]_p & X \ar[d]^q \\
S' \ar[r]^j & S} \]
where $k = u \circ i$, is a Cartesian square and $q$ is flat. Since $j_*
\cF$ is perfect on $S$, we get that $q^* j_* \cF \cong
k_* p^* \cF \cong u_* \Phi_Z$ is perfect on
$X$.

Using also that $p^*$
and $i_*$ take perfect complexes to perfect complexes, we obtain an
induced functor
\[ \bar{\Phi}_Z: \rDsg(S' \subseteq S) \to \rDsg(Y \subseteq X),
\]
and this functor is fully faithful by \cite[1.1]{MR2437083}.

We now show that 
 if $\cG$ is
in $\rperf Y X$ then
$\Phi_Z^*( \cG )$ is in $\rperf {S'} S$. We have 
\[
j_* \Phi_Z^*( \cG ) = j_* \R p_*
i^{\flat} \cG \cong \R q_* u_* i_*( \mathbb{L} i^* \cG \otimes
\omega_{Z/Y}) \cong \R q_* u_*( \cG \otimes i_* \omega_{Z/Y}) 
\]
where the last isomorphism uses the projection formula.
Using the isomorphism
$u^* u_* \cG \cong \cG$ and the projection formula again, we obtain
\[
u_*( \cG \otimes
i_* \omega_{Z/Y})\cong u_*( u^* u_* \cG \otimes i_* \omega_{Z/Y})
\cong u_* \cG \otimes u_* i_* \omega_{Z/Y}.
\]
Since $\omega_{Z/Y}$ is locally free on $Z$
 and $i$ and $u$ have finite flat dimension,
$u_* i_* \omega_{Z/Y}$ is perfect. Since we are assuming
$u_* \cG$ is
perfect, this shows that 
$u_*( \cG \otimes i_* \omega_{Z/Y})$ is also perfect. It now follows from 
\cite[III.4.8.1]{MR0354655} that
$j_* \Phi_Z^*( \cG )  \cong \R q_* u_*( \cG
 \otimes i_* \omega_{Z/Y})$ is perfect.

Since
$\Phi_Z^*$ sends objects of  $\rperf Y X$ to objects of $\rperf {S'}
S$, we obtain an induced functor
\[ 
\bar{\Phi}^*_{Z}: \rDsg(Y \subseteq X) \to \rDsg(S' \subseteq S)
\]
that is right adjoint to $\bar{\Phi}_Z$.

By Lemma \ref{kernel_is_zero} below,
$\bar{\Phi}^*_{Z}$ has trivial kernel, and the rest of the proof is a formality. We have a pair of adjoint
functors
\[ \xymatrix{ \rDsg(S' \subseteq S)\ar@/^1pc/[rrr]^{\bar{\Phi}_Z} &&& \rDsg(Y \subseteq X)
  \ar@/^1pc/[lll]^{\bar{\Phi}_{Z^*}}} \]
such that one is fully faithful and the kernel of the other is
zero, and so the pair must be mutually inverse equivalences; see, for
example the
proof of Theorem 2.1 in \cite{MR2437083}.
\end{proof}

\begin{lem}
  \label{kernel_is_zero}
For $\cG \in \rDsg(Y \subseteq X)$, if
$\bar{\Phi}_{Z^*}( \cG ) \cong 0$ then $\cG \cong 0$.
\end{lem}

\begin{proof} We need to prove that if $\cG \in \rrperf(Y \into X)$ is
  such that $\Phi_Z^*(\cG)$ is perfect on $S'$, then $\cG$ is perfect
  on $Y$. 
By \cite[6.6]{BW11a} we may assume that $\cG$ is the coherent sheaf on
$Y$ given by the cokernel of an
object of $MF(X, \cO(1), W)$; see Section \ref{setup} for the
definition. In particular this implies, by \cite[5.2.1]{BW11a}, that $\cG$ has an infinite right
resolution by locally free coherent sheaves on $Y$. It follows from
this that
$\mathbb{L} i^* \cG \cong i^* \cG$ and hence $\Phi_Z^*( \cG ) \cong \R p_*(i^*
\cG \otimes \omega_{Z/Y})$.  

Using the short exact sequence
\eqref{normal_bundle_ses} and \cite[p. 139]{MR0222093} we have that
$\omega_{Z/Y} \cong \Lambda^{r} p^* j^* \cE \otimes \cO_Z(-1)$. The
projection formula 
gives 
\[ 
\R p_*(i^*
\cG \otimes \omega_{Z/Y}) \cong \R p_*( i^* \cG(-1)) \otimes
\Lambda^{r} j^* \cE,
\]
and hence $\R p_*( i^* \cG(-1))$ is perfect on $S'$, since $\Lambda^{r}
j^* \cE$ is a line bundle.
By the definition of $MF(X, \cO(1), W)$ and the equivalence $[MF(X,
\cO(1), W)] \cong \rDsg(S'
\subseteq S)$ of Theorem~\ref{our_equiv}, one sees that $\cG(n) \cong \cG[2n]$ in $\rDsg(S'
\subseteq S)$ for all $n \in \Z$. Thus $\R p_*(
i^* \cG(n))$ is perfect for all $n \in \Z$. It now follows from Lemma
\ref{perfect_iff_all_twists} below  that $i^* \cG$ is perfect on
$Z$.

Note that for any $x \in Y$, $\cG_x$ has an infinite right resolution by
free $\cO_{Y,x}$-modules. Thus we may assume that we are the following
situation: let $A$ be a commutative local ring, $f_1,
\ldots, f_c$ an $A$-regular sequence, set $B = A/(f_1), C =
B/(f_2, \ldots, f_c)$, and let $M$ be a finitely generated $B$ module with an
infinite free right resolution such that $M \otimes_B C$ has finite
projective dimension over $C$. To finish the proof it is enough to show
that under these conditions $M$ must be free. By
\cite[Lemma 0.1]{Ei80}, the $B$-regular sequence $(f_2, \ldots, f_c)$
is also $M$-regular. Thus by \cite[4.3.12]{We94} we have $\pd_B M = \pd_C (M
\otimes_B C)$. Finally, since $C$ has finite projective dimension over
$B$, one checks that the infinite free right resolution of $M$ remains
exact upon tensoring by $C$ and so $M \otimes_B C$
also has an infinite right free resolution. Thus $M \otimes_B C$ is
free, and hence so is $M$ because they have the same projective dimension.
\end{proof}

\begin{lem}{\cite[Lemma 2.6]{MR2437083}}
 \label{perfect_iff_all_twists}
 An object $\cG \in \Db( Z )$ is perfect if and only if $\R p_*(
 \cG(n))$ is perfect over $S'$ for all $n \in \Z$.
\end{lem}

Since we are not working over a field, the argument given in
\emph{loc.\ cit}.\ does not apply.
\begin{proof}
If $\cG$ is perfect, then so is $\cG(n)$ for all $n$, and by \cite[III.4.8.1]{MR0354655} so is $\R p_*( \cG(n) )$.

To show the converse we may work locally and assume $S'$ is affine, so
that $p$ has the form
$p: Z = \bP^r_R \to \Spec R$.  Since $\cG$ is perfect if and only if some syzygy
in a locally free resolution of $\cG$ is, we may assume that $\cG$ is
a coherent sheaf. Moreover, since $\cG$ is
perfect if and only if $\cG(n)$ is
perfect for some $n$, we may assume that $\rH^i( Z, \cG(n) ) = 0$ for
all $i > 0$ and $n \geq 0$ by Serre's Vanishing Theorem. In particular this implies that $\cG$ is
regular and that $\R p_* \cG \cong p_* \cG$. 
(See Section \ref{sec_proj_res} where the definition is of regularity is
recalled.)

Since $\cG$ is regular, \cite[8.1.11]{MR0338129} shows that there is a resolution of $\cG$ of the form
\begin{equation}
0 \to \cO_X(-r+1) \otimes_R T_{r-1} \to \cdots \to \cO_X \otimes_R
T_0 \to \cG \to 0\label{eq:0-to-co_x}
\end{equation}
where $T_i$ are $R$-modules defined as follows. We set
$T_0$ to be $p_* \cG$, which is perfect by assumption, and define
$Z_1$ to be the kernel of the map $\cO_X(-1) \otimes_R T_0 \to
\cG$. Note that $p_*( Z_1(n) )$ is perfect for all $n$. Inductively
set $T_i = p_*( Z_{i-1}(i) )$ and $Z_i$ to be the kernel of
the canonical map $\cO_X(-i) \otimes_S T_i \to
Z_{i-1}$. By \cite[p. 132]{MR0338129} $Z_{i-1}(i)$ is regular.  Using this, it
follows from the definition and induction that each $T_i$ is perfect over $R$. Thus each term in
the resolution \eqref{eq:0-to-co_x} of $\cG$ is perfect over $Z$ and so $\cG$ must be as well.
\end{proof}

\section{Stable Ext-modules and complete resolutions}
\label{compl_resl_sec}
Here we show that if $M$ is a complex with bounded
finitely generated cohomology over a commutative Noetherian ring $A$,
and $M$ has a complete resolution in the sense of \cite{Ve06}, then
one may compute $\sExt^n_A( M , - ) := \Hom {\Dsg(A)} {\Sigma^{-n} M} -$ using the complete
resolution. This is well-known in the case $A$ is Gorenstein by
\cite{Bu87}, and is no-doubt known to the experts in this
generality. However we could not find the result in the literature.

\begin{defn}
  Let $M$ be a complex of $A$-modules with bounded and finitely
  generated cohomology. A \emph{complete resolution of $M$} is a diagram
\[ T \xra{\gamma} P \xra{\delta} M \]
where $T$ is an acyclic complex of projective $A$-modules such that $\Hom A
T A$ is also acyclic (i.e.\ $T$ is \emph{totally acyclic}), $\delta$ is a projective resolution (i.e.\ a
quasi-isomorphism such that $P$ a complex of projective $A$-modules with
$P^i = 0$ for $i \gg 0$), and $\gamma^i$ (the degree $i$ component of 
$\gamma$) is an isomorphism for $i \ll 0$.
\end{defn}

\begin{rem}
  The previous definition was made in \cite{Ve06} for arbitrary complexes; see also \cite{AvMart}.
\end{rem}

\begin{ex}
Let $R = Q/(f_1, \ldots, f_c)$ be as in \S \ref{assumptions}, and let $M$
be a finitely generated $R$-module with finite projective dimension
over $Q$, i.e.\ $M$ is in $\rDsg( Q \onto R )$. By \cite[1.2.10, 2.2.8]{Ch00} $M$ has finite G-dimension over
$R$,  and by \cite[2.4.1. 3.6]{Ve06} this implies that $M$ has a
complete resolution in the sense above.
\end{ex}

\begin{lem} \label{Lem21}
Let $M$ and $N$ be complexes of $A$-modules with bounded
  finitely generated cohomology and $f: M \to N$ a morphism of chain
  complexes such that $\cone f$ is perfect. If $T \xra{\gamma} P \xra{\delta}
  N$ is a complete resolution of $N$, then
  there exists a complete resolution $T \xra{\gamma'} P' \xra{\delta'} M$  of
  $M$ and a
  diagram 
\[ \xymatrix{ T \ar[r]^{\gamma'} \ar[d]_= & P' \ar[r]^{\delta'} \ar[d] & M \ar[d]_f \\
T \ar[r]^{\gamma} & P \ar[r]^{\delta} & N
} \]
that commutes up to homotopy.
\end{lem}

\begin{proof}
Consider the triangle 
\[
M \xra{f} N \to \cone f \to M[1]
\] in
$\mathsf{K}(A)$, the homotopy category of chain complexes of
$A$-modules. Let $\delta'': F \xra{\simeq} \cone f$ be a quasi-isomorphism from a
bounded complex of finitely generated projective $A$-modules, which
exists since $\cone f$ was assumed perfect. Let $g: P \to F$ be
a lifting over $\delta''$ of the composition $P \xra{\delta} N \to \cone f$. 
We can complete this to a map of triangles:
\[ \xymatrix{ \cone g[-1] \ar[d]_{\delta'} \ar[r] & P \ar[r]^g \ar[d]_\delta^\simeq & F \ar[r] \ar[d]^\simeq_{\delta''} & \cone g \ar[d] \\
M \ar[r]^f & N \ar[r] & \cone f \ar[r] & M[1]
}\]
Set $P' := \cone g[-1]$.
By construction $P'$ is a complex of projective $A$-modules
with $(P')^i = 0$ for $i \gg 0$ and $(P')^i = P^i$ for $i \ll
0$. Moreover $\delta'$ is a quasi-isomorphism since $\delta$ and
$\delta''$ are. Thus
$\delta'$ is a projective resolution of $M$.

Since $\Hom A T A$ is acyclic, and since
$F$ is a bounded complex of finitely generated projective
$A$-modules, it follows that $\Hom A T F$ is acyclic, i.e.\ there are
no maps from $T$ to $F$ in $\mathsf{K}(A)$. Thus the map
$\gamma: T \to P$ lifts, up to homotopy, to a map $\gamma': T \to
P'$. This is the complete resolution we seek.
\end{proof}

\begin{ex}
\label{mf_compl_resl_ex}
 Let $B$ be a local Noetherian commutative ring and $f
 \in B$ a
 non-zero divisor. Set $A = B/(f)$. By \cite[Theorem 1]{1101.4051} (see also Theorem
 \ref{our_equiv} above) there is a fully faithful functor
\[\coker:[MF(\Spec B, \cO_B, f)] \into \Dsg(A).\] Let $M$ be a complex of $A$-modules
 with bounded and finitely generated cohomology, and assume that $\bE =
 (B^n \xra{\phi} B^n \xra{\psi} B^n)$
 is a matrix factorization of $f$ such that $\coker \bE \cong M \in
 \Dsg(A)$. Set
 $$
T := \cdots \to A^n \xra{\phi \otimes_B A} A^n \xra{\psi \otimes_B
   A} A^n \to \cdots.
$$
It is clear that $T \to T^{\leq 0} \to \coker \bE$
 is a complete resolution. By Lemma \ref{Lem21}  there exists some projective resolution $P
 \to M$ which fits into a complete
 resolution $T \to P \to M$.
\end{ex}

When $A$ is Gorenstein the following is contained in \cite[6.1.2]{Bu87}.
\begin{lem}
\label{stable_ext_by_compr}
Let $A$ be a commutative Noetherian ring and let $M$ be an $A$-complex
with bounded finitely generated cohomology and a complete resolution $T
\xrightarrow{\gamma} P \to
M.$ Let $N$ be any $A$-complex with bounded finitely
generated cohomology. There is an isomorphism, natural in $N$,
\[ \eta^q_M : \rH^q \Hom A T N \xra{\cong} \sExt^q_A( M , N ). \]
Moreover, there is a commutative diagram
\[ \xymatrix{ \rH^q \Hom A P N \ar[rr]_{\rH^q \Hom A \gamma - } \ar[d]_\cong && 
  \ar[d]_\cong^{\eta^q} \rH^q \Hom A T N\\
\ar[rr] \Ext q A M N&& \sExt_A^q(M, N )
} \]
where the lower horizontal map is \eqref{eq:normal_to_stable}.
\end{lem}

\begin{proof}
  First note that we may assume
  that $q = 0$. We view the bounded derived category of $A$-modules as the homotopy
  category of complexes of finitely generated projective modules
  with bounded cohomology. Thus we replace $M$ by $P$ and we may
  assume that $N$ is a complex of finitely generated projective
  $A$-modules. Fix an integer $k$ such that
  $P^i = 0 = N^i$ for $i \geq k$. 

There is a natural map $T \to T^{\leq k}$, and the map $\gamma: T \to
P$ factors as
\[ \xymatrix{ T \ar[dr]^{\gamma} \ar[d] & \\
T^{\leq k} \ar[r]^\epsilon & P} \]
Note also that the cone of $\epsilon$ is perfect. 
In the remainder of this proof, we use $\simeq$ to denote a 
morphism whose cone is perfect. For example, we have $\epsilon:
T^{\leq k} \xra{\simeq} P$.

The natural map $T \to T^{\leq k}$ induces an isomorphism $\Hom A
{T^{\leq k}} N \xra{\cong} \Hom A T N$. Given an element $f$ in $Z^0 \Hom A
{T^{\leq k}} N$, we send it to the element of $\Hom {\Dsg(A)} P N$
represented by $P \xla{\epsilon}   T^{\leq k} \xra{f} N$.
If $f$ is null-homotopic
then this diagram represents the zero morphism. Thus we have a
well-defined map
\[ \eta^0_M: \rH^0 \Hom A T N \cong \rH^0 \Hom A
{T^{\leq k}} N \to \Hom {\Dsg(A)} P N = \sExt^0_A( M , N ).\] This map is independent of
the choice of $k$ and thus is functorial in $N$.

We now show that the map fits into the commutative diagram
above. Given an element $h \in \rH^0 \Hom A P N \cong \Ext 0 A P N$,
its image in $\Hom {\Dsg(A)} P N$ is represented by $P \xla{=} P
\xra{h} N$. In the other direction, it is sent to $h \circ \gamma$ in
$\rH^0\Hom A T N$. By definition $\eta^0$ sends this to $P \xla{\epsilon}
T^{\leq k} \xra{ h \circ \epsilon} N$.
The diagram
\[ \xymatrix{ & P \ar[dl]_= \ar[dr]^h & \\
P & & N \\
& T^{\leq k} \ar[ul]_{\epsilon} \ar[uu]_{\epsilon} \ar[ur]_{h \circ \epsilon
  }
} \]
shows that the diagram commutes.

We now show that $\eta^0$ is an isomorphism. Suppose that
$\eta(f) = 0$. Then there exists a complex of projective modules $X$ and a diagram
\[ \xymatrix{ & T^{\leq k} \ar[dl]_\epsilon^{\simeq} \ar[dr]^f & \\
P & & N \\
& X \ar[ul]_{\simeq} \ar[uu]_{\simeq} \ar[ur]_{g \sim 0}
} \]
in the homotopy category. Clearly $T \to T^{\leq k} = T^{\leq k}$ is a complete
resolution. By the previous lemma, there exists a complete resolution $T \to
X' \to X$ making the following diagram commute:
\[ \xymatrix{ T \ar[r] & T^{\leq k} \ar[r]^= & T^{\leq k} \ar[dr]^f & \\
T \ar[u]^= \ar[r] & X' \ar[r] \ar[u] & X \ar[u]^\simeq \ar[r]_{g \sim 0} & N
} \]
In other words we have a commutative diagram
\[ \xymatrix{ T \ar[r] \ar[dr]_{\sim 0} & T^{\leq k} \ar[d]^f \\
& N } \]
This shows that $f$ is null-homotopic and so $\eta$ is injective.

Now let $P \xla{\simeq} X \to N$ be any morphism from $P$ to $N$ in
$\Dsg(A)$. By the previous lemma there exists a complete
resolution $T \to X' \to X$ and a commutative diagram
\[ \xymatrix{ T \ar[r] \ar[d]_= & X' \ar[r] \ar[d] & X \ar[d]_\simeq \\
T \ar[r] & P \ar[r]_= & P
} \]
{}From this, one sees that there exists an integer $l$ and a commutative
diagram
\[ \xymatrix{ & X \ar[dr] \ar[dl]_\simeq & \\
P & & N \\
& T^{\leq l} \ar[uu] \ar[ul]^\simeq \ar[ur]
} \]
which shows that $P \xla{\simeq} X \to N$ is in the image of $\eta$.
\end{proof}


\bibliographystyle{plain}
\renewcommand{\baselinestretch}{1.1}
\renewcommand{\MR}[1]{%
  {\href{http://www.ams.org/mathscinet-getitem?mr=#1}{MR #1}}}
\providecommand{\bysame}{\leavevmode\hbox to3em{\hrulefill}\thinspace}
\newcommand{\arXiv}[1]{%
  \relax\ifhmode\unskip\space\fi\href{http://arxiv.org/abs/#1}{arXiv:#1}}


\end{document}